\documentclass[10pt, reqno]{amsart}

\usepackage{epsfig,latexsym,amsfonts,amssymb,amsmath,amscd,graphics,epic}
\usepackage{amsfonts,amssymb,amsmath,amscd,amsthm}
\usepackage[mathscr]{eucal}
\usepackage{mathrsfs}

\usepackage{mathbbol}

\usepackage{oldgerm,units}
\usepackage{wrapfig,epsfig}

\usepackage{ifthen}
\usepackage{mathbbol}

\usepackage{amsthm}
\usepackage[mathscr]{eucal}
\usepackage{mathrsfs}
\usepackage{mathbbol}
\usepackage{oldgerm,units}
\usepackage{wrapfig}

\newtheorem{theorem}{Theorem}[section]
\newtheorem{proposition}[theorem]{Proposition}
\newtheorem{definition}[theorem]{Definition}
\newtheorem{kdefinition}[theorem]{(Key) Definition}

\newtheorem{lemma}[theorem]{Lemma}

\newtheorem{notation}[theorem]{Notation}
\newtheorem{corollary}[theorem]{Corollary}

\newtheorem{example}[theorem]{Example}
\newtheorem{remark}[theorem]{Remark}

\newtheorem{question}[theorem]{Question}

\newtheorem*{theorem*} {Theorem}
\newtheorem*{corollary*} {Corollary}
\newtheorem*{kquestion*} {Known question}
\newtheorem*{question*} {Question}
\newtheorem*{remark*}{Remark}
\newtheorem*{example*}{Example}

\newcommand{\Real}{\mathbb R}

\newcommand{\Fld}{\mathbb K}

\newcommand{\one}{\mathbb{1}}
\newcommand{\zero}{\mathbb{0}}

\newcommand{\Trop}{\mathbb T}

\newcommand{\trop}[1]{\mathcal{#1}}

\newcommand{\tB}{\trop{B}}
\newcommand{\tC}{\trop{C}}

\newcommand{\tG}{\trop{G}}
\newcommand{\tH}{\trop{H}}

\newcommand{\tT}{\trop{T}}

\newcommand{\tW}{\trop{W}}

\newcommand{\al}{\alpha}

\newcommand{\sig}{\sigma}

\newcommand{\rnk}{\operatorname{rk}}

\newcommand{\OP}{\left(}
\newcommand{\CP}{\right)}

\newcommand{\Cl}[1]{#1^\bullet}

    \ifx\proof\undefined
    \newenvironment{proof}{
    \smallskip
    \noindent\emph{Proof.}}{\hfill\(\Box\)
    \bigskip
    } \fi

\newcommand{\vMat}[4]{\OP \begin{array}{ll}
  #1 & #2 \\
  #3 & #4
\end{array}\CP}

\newcommand{\vvMat}[9]{\small{\OP \begin{array}{ccc}
  #1 & #2 & #3\\
  #4 & #5 & #6\\
  #7 & #8 & #9\\
\end{array}\CP}}

\newcommand{\ifdef}[3]{\ifthenelse{\equal{#1}{true}}{#2}{#3}}

\numberwithin{equation}{section}

\input xy
\xyoption{all}

\usepackage{rotating}
\usepackage{amssymb,  mathrsfs, color,  rotating}

\input amssymb.sty
\newcommand{\ds}[1]{\ {#1} \ }

\def\mfa{\frak a}

\def\nucong{\cong_{\nu}}

\newcommand\cl[2]{\rwcl{#1}{\str}{#2}}

\newcommand\clrw[3]{#1[#3,#2]}
\newcommand\rwcl[3]{#1[#2,#3]}

\def\one{\mathbb 1} \def\zero{\mathbb 0}
\def\rone{\one} \def\rzero{\zero}
\def\len{\ell}
\def\pth{p}

\def\str{\, \ast \,}

\newcommand\thmref[2]{\pSkip\textbf{Theorem #1. }\emph{#2}\pSkip}

\newcommand\corr[1]{\pSkip\textbf{Corollary. }\emph{#1}\pSkip}

\def\pSkip{{\vskip 1.5mm}\noindent}
\def\ltw{0.7\textwidth}
\newcommand\HH{\mathscr{H}}
\newcommand\MM{\mathscr{M}}

\def\semifield{semifield}
\def\semiring{semiring}

\def\iplus{\; \widetilde{+} \;}
\def\idot{\; \widetilde{\cdot} \;}

\newcommand\boxtext[1]{\pSkip \qquad \qquad \qquad \framebox{\parbox{\ltw}{#1}}\pSkip}

\def\gV{V_G}
\def\gE{E_G}

\def\Pow{\operatorname{Pw}}
\def\iff{\Leftrightarrow}
\def\imp{\Rightarrow}
\def\bF{\mathbb F}
\def\bK{\mathbb K}
\def\bN{\mathbb N}

\def\blA{A_{\bool}}
\def\blB{B_{\bool}}

\def\blD{D_{\bool}}
\def\sbA{A_{\sbool}}
\def\sbB{B_{\sbool}}
\def\sbC{C_{\sbool}}
\def\sbD{D_{\sbool}}

\def\fA{A_{\bF}}
\def\f2A{A_{\bF_2}}
\def\blA{A_{\bool}}
\def\ffA{A_{F}}
\def\kA{A_{\bK}}
\def\fB{B_{\bF}}
\def\fD{D_{\bF}}

\def\adj{\operatorname{adj}}
\def\inc{\operatorname{inc}}

\def\adA{A_{\adj}}
\def\inA{A_{\inc}}

\def\Gr{G}
\def\dGr{G}

\def\H{\HH}
\def\M{\MM}
\def\inM{\M_{\inc}}

\def\Cl{\operatorname{Col}}
\def\Rw{\operatorname{Row}}

\def\sm{\setminus}

\def\1{1^\nu}

\def\ntH{\tH^{\operatorname{c}}}

\newcommand{\etype}[1]{\renewcommand{\labelenumi}{(#1{enumi})}}
\def\eroman{\etype{\roman}}
\def\ealph{\etype{\alph}}

\def\({\left(}
\def\){\right)}

\def\tGz{{\tG_0}}

\def\tTzB{\{0,1\}}

\def\bool{\mathbb B}

\def\sbool{{\mathbb{SB}}}

\newcommand{\per}[1]{\operatorname{per}({#1})}

\begin{document}


\title[New Representations of Matroids and Generalizations]
{New Representations of Matroids and Generalizations}


\author{Zur Izhakian}

\address{   School of Mathematical Sciences, Tel Aviv
     University, Ramat Aviv,  Tel Aviv 69978, Israel.
\vskip 1pt
     Department of Mathematics, Bar-Ilan University, Ramat-Gan 52900,
Israel.
    }
    \email{zzur@math.biu.ac.il}

\thanks{The research of the first author has  been  supported  by the
Israel Science Foundation (ISF grant No.  448/09).}

\thanks{The research of the first author  has been supported  by the
Oberwolfach Leibniz Fellows Programme (OWLF), Mathematisches
Forschungsinstitut Oberwolfach, Germany.}

\author{John Rhodes}
\address{Department of Mathematics, University of California, Berkeley,
970 Evans Hall \#3840, Berkeley, CA 94720-3840 USA.}
\email{blvdbastille@aol.com;rhodes@math.berkeley.edu}

\thanks{\textbf{Acknowledgement:}
The authors would like to thank Professor Benjamin Steinberg for
detailed comments following careful reading of the manuscript.
Part of this work was done during the authors' stay at the
Mathematisches Forschungsinstitut Oberwolfach (MFO);  the authors
are very grateful to MFO for the hospitality and excellent working
conditions.}


\date{\today }




\begin{abstract} We extend the notion of
matroid representations by matrices over fields and consider new
representations of matroids by matrices over finite semirings,
more precisely over the boolean and the superboolean semirings.
This idea of representations  is generalized naturally to include
also hereditary collections. We show that a matroid that can be
directly decomposed as matroids, each of which is representable
over a field, has a boolean representation, and more generally
that any arbitrary hereditary collection is
superboolean-representable.
\end{abstract}

\maketitle



\section{Introduction}

Traditionally, matroids have been represented by using matrices
defined  over fields
\cite{OrientedMatroids,murota,oxley:matroid,qtheory,
tutte2,White,whitney}, mainly finite fields
\cite{Revyakin2000,Whittle}, or partial fields
\cite{SempleWhittle}; matroids that do have such a representation
are termed field-representable. It is well known that not every
matroid is field-representable; one of the most celebrated
examples for such a non-representable matroid over fields is the
direct sum of the Fano and the non-Fano matroids (see
\cite[Corollary 5.4.]{Oxley03whatis}).
 Over the years much effort has been invested in the attempt to
 specify families of matroids that are field-representable, this has been especially studied  with
 respect to the characteristic of a ground field used for constructing the matroid representation.
%

In this paper we introduce the idea of replacing the customarily
ground structure of the field one uses for representations of
matroids and consider instead  representations of matroids by
matrices
 over \semiring s; in particular over a certain
3-element supertropical \semiring  \
\cite{IzhakianRowen2007SuperTropical}, that is the
\emph{superboolean \semiring} \ $\sbool$. This \semiring \ is a
``cover'' of the   boolean \semiring, defined over the element set
$\sbool := \{ 1, \1,0 \}$, and its arithmetics is a modification
of the familiar boolean algebra (see \S\ref{sec:superboolean}).
Although the lack of negation, the superboolean   structure allows
natural algebraic analogs of classical notions such as dependence
of vectors and singularity of matrices which are so important for
a representation theory. These notions lead naturally to the key
setting of \emph{vector hereditary collections} (cf. Definition
\ref{defn:VecHC}) which are at the heart of our representation
approach.

A matroid
that has a representation by a superboolean matrix (i.e., is
isomorphic to a vector hereditary collection) is said to be
superboolean-representable. Using this concept of representations,
we show that in a sense all matroids are ``super-regular'',
namely, all matroids are superboolean-representable. It turns out
that this representation concept  is much broader and is feasible
not only for matroids but also for (finite) hereditary collections
-- a more general set-theoretic objects known also as abstract
simplicial complexes
\cite{QuillenK-Theory,QuillenHomotopy,Spanier}. One of our main
theorems in this paper is the following:
\thmref{\ref{thm:hdCol}}{Any hereditary collection is
superboolean-representable.}
%
The proof of this theorem shows an explicit simple construction of
such superboolean representations.

Focusing on boolean representations, these are representations
determined by matrices having only $1,0$ entries, we prove:
\thmref{\ref{thm:boolFRep}}{Any field-representable matroid is
also boolean-representable.}
We also provide an explicit algorithm to produce the matroid's
boolean-representation from its field representation. More
generally, we extend this result to obtain the following:
\thmref{\ref{thm:matDecom}}{Matroids that are directly
decomposable into field-representable matroids also have boolean
representation.}
Having this representation approach, matroids that are not
representable over fields, for example, as mentioned above,  the
direct sum of the Fano matroid and the non-Fano matroid, can be
represented by boolean matrices, cf. \S\ref{ssec:Fano}.

As expected,  our new representation ideas and the development
along this paper pave the way to a new set of open questions
stated in~\S\ref{sec:openQ}.

The appendixes indicates the generalization of these
representation ideas to matroid representations by matrices that
take  place over an arbitrary supertropical semiring
\cite{IzhakianRowen2007SuperTropical}. A major example for such a
 semiring is the extended tropical semiring
\cite{zur05TropicalAlgebra}, which is a ``cover'' of the standard
tropical (max-plus) semiring \cite{ABG,Develin2003,IMS,Sturm3},
having a much richer algebraic structure that allows the carry of
systematic theory of tropical linear algebra
\cite{zur05TropicalAlgebra,IzhakianRowen2008Matrices,IzhakianRowen2009Equations,IzhakianRowen2010MatricesIII,
IzhakianRowen2009TropicalRank}.

Using a trivial embedding of the superboolean semiring in any
supertropical semiring, we conclude in Appendix B that:
\corr{Every hereditary collection is $R${-representable}, for any
supertropical semiring $R$.}

This paper presents only the preliminary results on boolean and
superboolean representations, to be developed further in the
future.

\section{Hereditary collections  and simplicial complexes}

\subsection{Hereditary sets}  Throughout this paper we always assume that the
\textbf{ground set}, denoted by $E$, is a fixed finite set.   We
write $|E|$ for the cardinality of $E$ and $\Pow(E)$ for the
\textbf{power set} of $E$, i.e., the set of all subsets (including
the empty set $\emptyset$) of $E$. In what follows, unless
otherwise is specified, we always assume that $|E| = n$, and thus
have $|\Pow(E)| = 2^n$. Subsets of $E$ of cardinality $k$ are
termed $k$-sets.

We use \cite{murota}, \cite{Oxley03whatis}, and
\cite{oxley:matroid} as general references, especially in regard
to matroid theory (matroids are presented in \S\ref{ssec:matroids}
below).

\begin{definition}\label{def:hereditary}
Let $E$  be a  set  and let $\tH \subseteq \Pow(E)$ be an nonempty
 collection of subsets $J$ of~$E$. The nonempty collection
$\tH$ is called \textbf{hereditary} if every subset $J'$ of any $J
\in \tH$ is also in $\tH$, more precisely: \boxtext{
\begin{enumerate} \eroman
    \item[HT1:] \ $\tH$ is nonempty, \pSkip

    \item[HT2:]  \ $J' \subseteq J$, $J \in \tH \ \imp \ J' \in
    \tH$.
\end{enumerate}}
(Hence, the empty set $\emptyset$ is also in $\tH$.)
The pair $\H := (E,\tH)$, with $\tH$ hereditary over $E$, is
called a \textbf{hereditary collection}.
\end{definition}
\noindent  Hereditary collections are also known in the literature
as \textbf{abstract simplicial complexes}
\cite{QuillenK-Theory,QuillenHomotopy,Spanier}.

The members of the collection $\tH$ are called the
\textbf{independent subsets} of $E$, and therefore the empty set
is considered independent. A subset $J \subseteq E$ which is not
contained in $\tH$ is called \textbf{dependent}. We denote the
collection of dependent subsets of $E$ by $$ \ntH : = \{ X
\subseteq E : X \notin \tH \},$$ i.e., $\ntH = \Pow(E) \sm \tH$.
(Clearly,   $\emptyset \notin \ntH$.)

A maximal  independent subset   (with respect to inclusion) of
$\H$ is called a \textbf{basis} of the hereditary collection $\H$.
The set of all bases of $\H$ is denoted as $\tB(\H) \subseteq \tH$
and termed the \textbf{basis set} of the hereditary collection
$\H$. Clearly, $\tB(\H)$ is canonically defined and by Axiom HT2
determines the hereditary collection $\H$ uniquely. Note that the
family $\H_t := (E,\tH_t)$ of hereditary collections with fixed
ground set $E$ and $\tH_t$ varying
 is in 1:1 correspondence  with the anti-chains of the lattice $(\Pow(E),
 \subseteq)$, cf. ~ \cite{Birkhoff}, given by $\H_t \to \tB(\H_t)$.

 A minimal subset (with respect to inclusion) of the collection
  $\ntH$ of the dependent subsets of $E$ is
called a \textbf{circuit}. We denote the collection of all
circuits of a hereditary collection $\H$ by $\tC(\H)$, i.e.,
$\tC(\H)\subseteq \ntH$. Then, the family $\H_t := (E,\tH_t)$ with
fixed ground set $E$ and $\tH_t$ varying is in 1:1 correspondence
with the anti-chains of the lattice $(\Pow(E),
 \subseteq)$ given now by $\H_t \to \tC(\H_t)$.

The \textbf{rank} $\rnk(\H)$ of a hereditary collection $\H$ is
defined to be the cardinality of the largest member of the basis
set $\tB(\H)$ of $\H$:
$$ \rnk(\H) := \max\{|B| \ :  \  B \in \tB(\H)\}. $$
In particular,  one always has $0 \leq  \rnk(\H)  \leq n$, and
$\rnk(\H) = 0$ iff $\tH = \{ \emptyset \}$.

\begin{example}\label{exp:1.2} Let us start with some elementary  examples of hereditary
collections.
\begin{enumerate} \ealph
    \item $\H = ( E, \{ \emptyset\})$ is a hereditary collection of rank $0$.
    (This is a matroid, to be defined below in Definition \ref{def:matroid}.) \pSkip

    \item $\H = (E, \Pow(E))$ is a hereditary collection (also a matroid) whose  basis set contains
    only the set~$E$, i.e., $\tB(\H) = \{ E \}$, and thus has rank
    $n$. \pSkip

    \item The \textbf{uniform hereditary collection} (also a matroid) $U_{m,n} := (E, \tH_{m,n})$,
    with $0 \leq m  \leq n$, $|E| = n ,$   is defined to
    have the collection of independent subsets
    $$ \tH_{m,n} := \{ X \subseteq E : |X| \leq m \},    $$
    and has rank $m$.

    Notice that the hereditary collections in (a) and (b)
    above can be written in this notation as $U_{0,n} =  ( E, \{ \emptyset\})$ and  $U_{n,n} =  ( E,
    \Pow(E) )$, respectively. We also have $U_{n-1,n} = (E, \Pow(E) \sm \{ E\} )$.

    \pSkip

    \item Let $E = \{ 1, 2, 3, 4 \} $ and let  $\H$ be the hereditary collection having the bases  $\{1,2,3 \}$,
    $\{2,3,4 \}$, $\{1,4 \}$. Hence, all the
    2-subsets of $E$  are independent and are members of $\tH$. (This example is not a matroid.)

    \item Consider the  hereditary collection over the ground set  $E = \{a,b,c,d\}$  with  the bases  $\{a,b\}$,
    $\{b,c\}$,$\{a,c\}$, and $\{b,d\}$,  corresponding to the  edges of the
    diagram
 $$
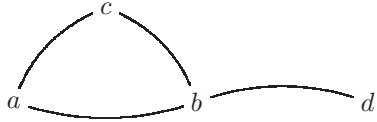
\xymatrix{
    & c      \ar@/^/@{-}[dr]   \ar@/_/@{-}[dl]\\
           a %
            \ar@/_/@{-}[rr]    &&      b
       \ar@/^/@{-}[rr] & & d
 }$$
 (This example is not a matroid.)
\end{enumerate}

\end{example}

The above examples provide some typical cases of hereditary
collections satisfying additional properties, to be discussed
later.

\begin{definition}\label{def:HCIso}  Hereditary collections $\H_1 = (E_1,
\tH_1)$ and $\H_2 = (E_2, \tH_2)$ are said to be
\textbf{isomorphic} if there exits a bijective map $ \varphi : E_1
\to E_2$ that respects dependence; that is $$\varphi(X_1) \in
\tH_2 \ \Leftrightarrow \ X_1 \in \tH_1,$$ for any $X_1 \subseteq
E_1$.
\end{definition}
\begin{definition}\label{def:dirctSum} The \textbf{direct sum} of two hereditary collections $\H_1 = (E_1,
\tH_1)$ and $\H_2 = (E_2, \tH_2)$, with disjoint nonempty ground
sets $E_1$ and $E_2$, is
$$ \H_1 \oplus  \H_2 := (E_1 \cup E_2, \{ J_1 \; \dot \cup \; J_2 : J_1 \in \tH_1, J_2 \in \tH_2
\}).$$
A hereditary collection $\H$ is \textbf{decomposable} if it can be
written as a direct sum  $\H = \H_1 \oplus \cdots \oplus \H_\ell $
of some hereditary collections $\H_i$ with disjoint nonempty
ground sets $E_i$'s, otherwise $\H$ is said to be
\textbf{indecomposable}.
\end{definition}

\subsection{Point replacement} We will impose various additional axioms on
hereditary collections.

In what follows, to simplify notation, given a subset  $X
\subseteq E$, and elements $x \in X$ and $p \in E$, we write $X-
x$ and $X + p$ for $X \sm \{ x\}$ and $X \cup \{ p \}$,
respectively; accordingly we write $X-x+p$ for
 $(X\sm \{x\}) \cup \{ p\}$. Abusing the terminology, we sometimes say
that an element $p \in E$ is independent iff $\{ p\}$ is
independent, i.e., $\{ p\} \in \tH$.

\begin{definition}\label{def:ratroid} We say that a hereditary collection $\H = (E,\tH)$ satisfies the \textbf{point replacement
property} iff
 \boxtext{
\begin{enumerate} \eroman
    \item[PR:]  For every  $\{ p\}  \in \tH$ and every nonempty subset $J \in \tH$ there exists
    $x \in J$ such that $J - x  + p \in \tH$.
\end{enumerate}}
\end{definition}

Notice that Examples \ref{exp:1.2}.(a)--(d) satisfy PR, while
example~(e) does not.  One also observes  by example~(d) that PR
does not imply that all the bases have the same cardinality.

\begin{proposition}\label{prop:het} The following are all equivalent for a hereditary collection
$\H$.
\begin{enumerate} \eroman
    \item Point replacement; \pSkip
    \item $\{p\}, X \in \tH$ with $p \notin X \neq \emptyset$
    implies $\exists x \in X$  such that $X - x + p \in
    \tH$; \pSkip
    \item $\{p\}, B \in \tH$  with $B \in \tB(\H)$
    implies $ \exists b \in B$ such that $B - b+p \in
    \tH$; \pSkip

    \item $\{p\}, B \in \tH$ with $B \in \tB(\H)$ and  $p \notin B$
    implies  $ \exists b \in B$  such that $B - b + p \in
    \tH$.

\end{enumerate}
\end{proposition}

\begin{proof} $(i) \iff (ii)$ ($(iii) \iff (iv)$)  since if $p \in X$
($b\in B$) chose  $x =p$ (b=p). \pSkip

$(i) \imp (iii)$ is trivial; so it  suffices  to prove $(iii) \imp
(i)$. Assume $\{p\} \in \tH$ and let $X \neq \emptyset$ be a
member of~$\tH$. Thus, there exists a basis  $B \in \tB(\H)$ that
contains $X$. Then, by $(iii)$,  $ \exists b \in B$, such that $B
-  b +p \in
    \tH$.
    If $b \in X $, then $X
- b + p  \in
    \tH$ and we are done. Otherwise, if $b \notin X$
    then $X + p  \subseteq B - b + p$ which is  a member of  $
    \tH$, and thus,  for any $x \in X$, $X - x + p \in
    \tH$ also lies in $\tH$. This implies $(iii) \Rightarrow (i)$.
\end{proof}

\begin{remark}\label{rmk:1.3c}
Define the \textbf{basis  replacement} condition as follows:
 \boxtext{
\begin{enumerate} \eroman
    \item[BR:]
 If $\{ p\}$ is independent and $B \in \tB(\H)$ then   $\exists b \in
B$ such that  $B - b + p$ is a basis.\label{con:13c}
\end{enumerate}}
By Proposition \ref{prop:het}.(iii), BR implies PR.  However BR is
not equivalent to point replacement; since Example
\ref{exp:1.2}.(d)  fulfills  PR, but does not satisfy~BR. (It
fails for $p=3$ and the basis $B = \{ 1, 4 \}$, since neither
$\{3,4\}$ nor $\{1,3 \}$ is a basis.)
\end{remark}
\subsection{Matroids}\label{ssec:matroids}
We now turn to the classical notion of matroids, cf.
\cite{Borovik_onexchange,Oxley03whatis,oxley:matroid}.

\begin{definition}\label{def:matroid}

A  \textbf{matroid} $\M$ is a pair $(E, \tH)$ with $\tH$
hereditary over the ground set $E$
 that  satisfies the following axiom:
 \boxtext{
\begin{enumerate} \eroman

    \item[MT:] If $I$ and $J$ are in $\tH$ and $|I| = |J| + 1$, then there exists $
    i \in I \sm J$ such that $J +i $ is in $\tH$.
\end{enumerate}}
\end{definition}

\begin{proposition}\label{1.4.b:prop}
The following properties, cf. \cite{Borovik_onexchange}, are
equivalent for a hereditary collection $\M = (E, \tH)$ to be a
matroid.

\begin{enumerate}\eroman
    \item \textbf{Exchange property (EP)}: $\forall A,B \in \tB(\M)$ and  $\forall a \in A\sm
    B$, $\exists b \in B \sm A $ such that $A - a + b$ is a basis of $\M$, i.e., it is
    an element of $
    \tB(\M)$. \pSkip
    \item \textbf{Dual exchange property (DEP)}: $\forall A,B \in \tB(\M)$ and $\forall a \in A\sm
    B$, $\exists b \in B \sm A $ such that $B -b + a  \in
    \tB(\M)$. \pSkip
    \item \textbf{Symmetric exchange property (SEP)}: $\forall A,B \in \tB(\M)$ and $\forall a \in A\sm
    B$, $\exists b \in B \sm A $ such that $B -b +a \in \tB(\M)$ and  $A -a + b \in
    \tB(\M)$.

\end{enumerate}

\end{proposition}

The proof of these equivalences, as well as the next lemma, are
standard in matroid theory, see \cite{oxley:matroid} and
\cite{Borovik_onexchange}.

\begin{lemma}[{\cite[Lemma 1.2.4]{oxley:matroid}}]\label{1.4.b:lem}
In a matroid $\M$ all the bases are have the same cardinality,
which is then equal  the rank of $\M$.

\end{lemma}

\begin{example} Consider the hereditary collection  $\H$ of Example \ref{exp:1.2}.(d).
The $2$-subset $\{ 1, 4\}$ is maximal in $\tH$ with respect to
inclusion, and thus is a basis of $\H$. Therefore,  since $\H$ has
rank $3$, $\H$ is not a matroid (recall that it does not satisfy
BR) but it satisfies PR.

\end{example}

\begin{proposition}\label{1.4.c:prop}
Any matroid satisfies the point replacement property PR (cf.
Definition \ref{def:ratroid}).
\end{proposition}
\begin{proof} We  assume the dual exchange property, cf.
Proposition~ \ref{1.4.b:prop}.(ii),  and the hypothesis of
Proposition~ \ref{prop:het}.(iv). Then we need to prove that for a
given basis $B \in \tB(\M)$ and a  point $p \in E$ there is an
element $b \in B$ such that $B - b + p$ is independent.

Pick a basis $A \in \tB(\M)$ containing $p$, set  $a = p \in A \sm
B $, and apply the dual exchange property, yielding that there is
$b \in B$ so that $B - b + p $ is a basis, hence independent.
\end{proof}

\begin{example} Consider the elementary examples of Example \ref{exp:1.2}.
\begin{enumerate} \ealph
    \item Example \ref{exp:1.2}.(d) satisfies PR is of rank $3$, and is not
a matroid. \pSkip

    \item
 Example \ref{exp:1.2}.(e) does not satisfy PR and thus
is not a matroid. (Take the element $d$ with respect to the
2-subset $\{a,c \}$.)
\end{enumerate}

\end{example}

\begin{proposition}  Let $\H = (E,\tH)$ be a hereditary collection of rank $2$.  Then $\H$
satisfies PR iff $\H$ is a matroid.
\end{proposition}
(Note that Example \ref{exp:1.2}.(e) shows that the proposition
fails for rank $ \geq 3$.)
\begin{proof}
$(\Rightarrow)$ Assuming that $\H$ is of rank $2$ and satisfies
PR, we show that $\H$ satisfies the dual exchange property.

First, we claim that all bases of $\H$ have cardinality $2$.
Indeed, pick $\{ p \} \in \tH$ and take $X \in \tH$ with $|X| =
2$, which exists by assumption. Suppose $X = \{ x_1, x_2 \} $.  If
$p \in X$ we are done. Otherwise, $p \notin X$ and the set $\{p,
x_1\}$ or $\{p, x_2\}$ is independent by PR, and thus is a basis
of $\H$  by maximality.

We next need to verify that the dual exchange propriety
(Proposition \ref{1.4.b:prop}.(ii)) is satisfied. Let $A = \{ a_1,
a_2\}$ and $B = \{ b_1, b_2\}$, where $A \neq B$ are two bases of
$\H$. If $A \cap B = \emptyset $ then PR implies the dual exchange
property for $A$ and $B$. On the other hand, if $A \cap B \neq
\emptyset$, say  $A = \{a,c \}$ and $B = \{b,c \}$, then the dual
exchange property is trivial by taking $a = b$.

$(\Leftarrow)$ By Proposition \ref{1.4.c:prop}.
\end{proof}
The answer to the next question is known to be false.

\begin{question*}\label{q:14.b}
Is the BR condition  equivalent to the MT axiom?
\end{question*}
MT implies  BR, with an easy proof similar to that of Proposition
\ref{1.4.c:prop}. The next section shows that the converse is
false.  (We know by Example \ref{exp:1.2}.(d)  that BR is a
stronger condition than PR.)

In the next couple of sections we consider  operations on
hereditary collections, resulting in new hereditary collections.
These operations  are standard for the case of matroids but are
somewhat less obvious for hereditary collections.

\subsection{Duality}

\begin{definition}\label{def:dual}  We define the \textbf{dual hereditary collection}
 $\H^* $ of a hereditary collection $\H = (E,
\tH)$ in terms of its bases as:
$$ \H^* := (E, \tH^*), \qquad B^* := E \sm B \in \tB(\H^*) \ \iff \  B \in \tB(\H).$$
 \end{definition}
Clearly, we  have $(\H^*)^* = \H$, for any hereditary collection
$\H$.

\begin{example}\label{1.5.b:exp} Consider Example \ref{exp:1.2}.(d), which satisfies
PR but is not a matroid. The bases of the dual hereditary
collection $\H^*$ are  $\{ 4 \}$, $\{ 1 \}$,  and $\{2,3 \}$ taken
over the same ground set  $E = \{1,2,3,4 \}$. Thus $\H^*$ which
does not satisfy PR, and therefore point replacement is not
preserved under duality.
\end{example}

\begin{proposition}[{\cite[Theorem 2.1.1]{oxley:matroid}}]\label{1.5.c:prop}
The dual of a matroid is a matroid.
\end{proposition}
\begin{proof} Immediate by  the exchange property and the dual exchange
property,  Proposition~\ref{1.4.b:prop}.(i) and
Proposition~\ref{1.4.b:prop}.(ii), respectively.
\end{proof}

\begin{example}\label{exp:1.5.d}
This is an example from James Oxley (private  communication)
showing that a hereditary collection $\H$ and its dual $\H^*$ can
both  satisfy PR (also BR) without $\H$ being a matroid.

Let $E = \{1,\dots,6 \}$ and consider the hereditary collection
$\H = (E,\tH)$ whose bases  are all 3-subsets of~$E$ except the
3-subsets  $A = \{1,2,3 \}$ and $B= \{1,3,4 \}$. These are not
bases of a matroid, since $A$ and $B$ are circuits, so $\{1,2,4
\}$ must contain a circuit which it does not. (See \cite[Chapter
1]{oxley:matroid} the weak circuit elimination axiom.) On the
other hand, BR is true for $\H$, since every 4-subset contains at
least two bases.

Similarly,   BR is true for $\H^*$ since every 2-subset is
contained in at least two bases of $\H$.
\end{example}

\begin{proposition}\label{1.5.e:prop} Condition BR for $\H$ is
implied  by PR together with the assumption that all the bases
have the same cardinality~$k$.
\end{proposition}
\begin{proof}
This is true since from Proposition \ref{prop:het}.(iii) if all
the bases have cardinality $k$, then any independent subset with
$k$ elements is a basis.
\end{proof}

\begin{question}\label{q:14.f} Is the converse of Proposition
\ref{1.5.e:prop} true?
\end{question}
\begin{remark}\label{rmk:1.5.g} $ $
\begin{enumerate}
    \item
When all the bases of a hereditary collection $\H$ have the same
cardinality $k$ (e.g., when  $\H$ is a matroid), then all the
bases of $\H^*$ are of cardinality $|E|-k$. \pSkip

\item If $\H$ satisfies PR and has rank 2, then by Proposition
\ref{1.5.c:prop} $\H^*$ is a matroid of rank $|E|-2$.

\end{enumerate}
\end{remark}

Given  a hereditary collection  $\H = = (E, \tH)$ of rank $2$ we
associate to $\H$ the graph $G := (\gV,\gE,)$ with  vertex set
$\gV = W$,  such that the bases of~ $\H$ are 2-subsets
corresponding to the edges of $G$. See Example \ref{exp:1.2}.(e).
The circuits of~$\H$ are 2-subsets corresponding to the missing
edges of $G$, and all subsets of $\gE$ that give complete
subgraphs on $3$ vertices.


\subsection{Deletion, contraction, and minors  }

\begin{definition}\label{def:deletion}
The \textbf{deletion} of a subset $X \subseteq E$ from a
hereditary collection  $\H = (E, \tH)$ is defined as
$$ \H \sm X := (E \sm X, \tH \sm X),$$
where $\tH \sm X := \{ Y \in \tH : Y \subseteq E \sm X  \}$.

The \textbf{contraction}, denoted  $\H / X$,  of $X$ is defined as
$(E \sm X, \tH / X)$, where $\tH / X$ is given  by:
$$ \begin{array}{ccc}
Y \in \tH / X  & \iff &  Y \cup B_X \in \tH \text{ for some $\H$
maximal independent subset }  B_X  \text { of } X.
   \end{array}
$$
\end{definition}

One sees that the empty set is contained in $\H / X$ and the
contraction of any  basis $B \in \tB(\H)$ gives the hereditary
collection $(E \sm B, \{ \emptyset\})$; while on the other hand
$\H / \emptyset = \H$.

\begin{remark}\label{1.6.c:rmk} $ $

\begin{enumerate} \eroman
    \item
In many applications the subset $X$ is assumed to be  independent
in $\H$, i.e., $X \in \tH$, so in this case $Y \in \tH / X$ if $X
\cup Y$ is independent in $\H$. \pSkip

    \item The definition of contraction does \textbf{not} satisfy $\H/ X = (\H^* \sm X)^*$
    for hereditary collections  in general, but does for matroids. \pSkip

    \item For disjoint  subsets $X$ and $Y$ of $E$, one has $$\H \ / \  Y \sm X = \H \sm X \ / \  Y,$$ as is easy to verify.
\end{enumerate}
\end{remark}
\begin{definition}\label{def:minor}
A \textbf{minor} $\H' \subseteq \H$ of a hereditary collection $\H
= (E, \tH)$ is a hereditary collection which is obtained from $\H$
by a sequence of deletions and contractions, which  is equivalent
to  $\H'$ being $$\H' = \H / X \sm Y = \H \sm Y / X$$ for some
disjoint subsets $X,Y \subseteq E$, where $\H = \H \sm \emptyset =
\H / \emptyset.$

A minor $\H' $ of $\H$ is said to be a \textbf{proper minor} if
$\H' \neq \H$.
\end{definition}

\begin{proposition}\label{1.6.c:prop} Given a hereditary collection $\H = (E, \tH)$.
Let $Y \subseteq  E$ be a subset of $E$, then for any $X \subseteq
E \sm Y$:
\begin{enumerate} \eroman
\item $X$ is independent in $\H \sm  Y$ iff $X$ is independent in
$\H$. \pSkip

\item $X$ is a circuit of $\H \sm  Y$ iff $X$ is a circuit of
$\H$. \pSkip

\item $X$ is a basis  of $\H \sm  Y$ iff $X$ is a maximal subset
of $E \sm Y$ that is independent in $\H$. \pSkip

\item $X$ is independent in $\H /  Y$ iff $X \cup B_Y $ is independent in
$\H$ for some subset $B_Y$ of $Y$ that is independent in $\H$.
\pSkip


\item $X$ is a basis of  $\H /  Y$ iff $X \cup B_Y$ is a basis of
$\H$ for some maximal subset $B_Y$ of $Y$ that is independent in
$\H$. \pSkip
\end{enumerate}
\end{proposition}

\begin{proof}
Straightforward.
\end{proof}
\begin{proposition}\label{1.6.f:prop} $ $
\begin{enumerate} \eroman
    \item Matroids are closed under taking  duals, deletions and contractions,
    and hence under minors. \pSkip

    \item  The class of hereditary collections  satisfying  PR is closed under
    deletion, but is not closed under contractions and duals, and hence not
    under minors. \pSkip

    \item A hereditary collection $\H$ is a matroid iff all minors of $\H$ satisfy PR.
\end{enumerate}
\end{proposition}

\begin{proof} (i): Standard, see \cite{Oxley03whatis} or
\cite{oxley:matroid}. \pSkip

(ii): That PR is preserved under deletion is clear. To see that is
not closed under taking duals,  consider the hereditary collection
$\H$ given in Example \ref{exp:1.2}.(d), which satisfies PR. Its
dual, given in Example \ref{1.5.b:exp},  shows that $\H^*$ does
not satisfy PR. Similarly, since $\H$ has bases $\{1,2,3\}$,
$\{2,3,4\}$, $\{1, 4\}$, then $\H / \{ 1 \}$ has the bases
$\{2,3\}$, $\{4\}$, which clearly do not satisfy PR. \pSkip

(iii): We must show that if all the minors of $\H$ satisfy PR,
then $\H$ satisfies the dual exchange property, cf. Proposition
\ref{1.4.b:prop}(ii). Given bases $A,B$, let $C = A \cap B$, $C$
is an independent set of $\H$. Consider $\H / C$, which must
satisfy PR. Now, by Proposition \ref{1.6.c:prop}  and Remark
\ref{1.6.c:rmk}.(i), we see that PR for $\H / C$ implies $\forall
a \in A- B$, $\exists b \in B- A$ such that $(B-A) - b + a$ is
independent in $\H /C$. So $((B- A) - b +a) \cup C = B-b +a$ is
independent in $\H$.

Therefore,  we have proved the following condition:
\begin{equation}\label{eq:*}
\text{$\forall A,B \in \tB(\H)$, $\forall a\in A-B$, $\exists b
\in B-A$ such that  $B-b + a$ is independent}. \tag{$*$}
\end{equation}
We will use $(*)$ to show that all the bases of $\H$ have the same cardinality.

If $A,B$ are different bases of $\H$, with $|A| < |B|$, then
applying $(*)$ inductively $|A|$ times  and extending $B-b +a$ to
a basis each time would imply  $A \subsetneqq B$, with $B$ a basis
-- a contradiction.

But if all the bases of $\H$ have the same cardinality, then
condition $(*)$ is the same as the dual exchange property.
\end{proof}

\section{Boolean and superboolean
algebras}\label{sec:superboolean}
 In this section all the proofs will be self-contained but see the references for further
 exposition (and generalizations).

The very well known  \textbf{boolean \semiring} is the two element
idempotent \semiring \ (see Appendix~A for the formal definition)
$$\bool := (\{ 0,1\}, + , \cdot \;),$$ whose addition and
multiplication are given by the following tables:
$$ \begin{array}{l|ll}
   + & 0 & 1  \\ \hline
     0 & 0 & 1  \\
     1 & 1 & 1 \\
   \end{array} \qquad \text{and} \qquad
\begin{array}{l|ll}
   \cdot & 0 & 1\\ \hline
     0 & 0 & 0  \\
     1 &  0 & 1  \\
   \end{array} \ .
   $$

The \textbf{superboolean \semiring}  $\sbool$  is the finite
supertropical \semiring \ \cite{IzhakianRowen2007SuperTropical}, a
``cover'' of the boolean \semiring, with the  three elements $$
\sbool : = (\{ 1, 0, 1^\nu \}, + , \cdot
 \;)$$
endowed with the two binary operations:
$$ \begin{array}{l|lll}
   + & 0 & 1 & 1^\nu \\ \hline
     0 & 0 & 1 & 1^\nu \\
     1 & 1 & 1^\nu & 1^\nu \\
     1^\nu &1^\nu & 1^\nu & 1^\nu \\
   \end{array} \qquad
\begin{array}{l|lll}
   \cdot & 0 & 1 & 1^\nu \\ \hline
     0 & 0 & 0 & 0 \\
     1 &  0 & 1 & 1^\nu \\
     1^\nu & 0 & 1^\nu & 1^\nu \\
   \end{array}
   $$
 addition and multiplication, respectively.
The superboolean \semiring \ $\sbool $ is totally  ordered by $$
1^\nu \
> \ 1 \
> \ 0 .$$
Note that $\sbool$ is \textbf{not} an idempotent \semiring,  since
$1 +1 = \1$, and thus  $\bool$ is \textbf{not} a subsemiring of
$\sbool$.

The element $\1$ is called the \textbf{ghost} element, where $\tGz
:= \{0, \1\}$ is the  \textbf{ghost ideal}\footnote{In the
supertropical setting, the elements of the complement of $\tGz$
are called \textbf{tangibles}.}  of $\sbool$.  Two elements $a$
and $b$ of $\sbool$ are $\nu$-\textbf{equivalent}, written $a
\nucong b$, if $a = b$ or $a, b \in \{1,  \1 \}$.


%
(Further details on supertropical \semiring \ structures are given
in Appendix A below. Full details  can be found in
\cite{IzhakianRowen2007SuperTropical}.)

\subsection{Superboolean matrix algebra}\label{ssec:matrixAlg}

The semiring $M_n(\sbool)$ of $n \times n$ superboolean matrices
with entries in ~$\sbool$ is defined in the standard way, where
addition and multiplication are induced from the operations of
$\sbool$ as in the familiar matrix construction. The unit element
$I$ of
 $M_n(\sbool)$, is the matrix with $1$ on the main diagonal and
whose off-diagonal entries are all $0$.

A typical matrix is often denoted as $A =(a_{i,j})$, and  the zero
matrix is  written as $(0)$. A matrix is said to be a
\textbf{ghost matrix} if all of its entries are in $\tGz$.  A
boolean matrix is a matrix with coefficients in $\tTzB$, the
subset of boolean matrices is denoted by $M_n(\bool).$

The following discussion is presented for superboolean matrices,
where boolean matrices are considered as superboolean matrices
with entries in $\tTzB$.
 Note that boolean
matrices $M_n (\bool)$ are \textbf{not} a sub-semiring of the
semiring of superboolean matrices $M_n(\sbool)$.

In the standard way, for any matrix $A \in M_n(\sbool)$, we define
the \textbf{permanent} of  $A = (a_{i,j})$ as:
\begin{equation}\label{eq:det1}
\per{A} := \sum _{\pi \in S_n}a_{\pi (1),1} \cdots a_{\pi
(n),n}\end{equation}
where $S_n$ stands for the group of permutations  of
$\{1,\dots,n\}$. Note that the permanent of a boolean matrix can
be $\1$.           We say that a matrix $A$ is
\textbf{nonsingular} if $\per{A} = 1$, otherwise $A$ is said to be
\textbf{singular}.
\begin{remark}\label{rmk:perma}
One major computational tool in tropical matrix theory is the
digraph of a matrix, we recall some basic definitions from
\cite[\S3.2]{IzhakianRowen2008Matrices}, restricted here for the
case of  superboolean semiring.

Given   an $n \times n$ superboolean matrix $A = (a_{i,j})$, we
associate the matrix $A$ with the \textbf{digraph} $\dGr _A =
(\gV, \gE)$ defined to have vertex set $\gV =\{ 1, \dots, n\}$,
and an edge $(i,j)$ from $i$ to $j$, labeled $a_{i,j}$, whenever
$a_{i,j} \ne 0$.


The \textbf{length} $\len(\pth)$ of a path~$\pth$ is the number of
edges of the path. An edge $(i,i)$ is  called a \textbf{self
loop}. A path is \textbf{simple} if each vertex appears only once.
A \textbf{simple cycle}
 is a simple path except that the starting vertex  and the terminating
vertex are the same. We define a $k$-\textbf{multicycle} $\sig$ in
a digraph to be the union of vertex  disjoint simple cycles, the
sum of whose lengths is $k$;  a $k$-multicycle $ \sig$ is labeled
$\1$ if one of its edges is labeled $\1$, otherwise $\sig$ is
labeled~$1$.

From  this graph view, each nonzero summand $a_{\pi (1),1} \cdots
a_{\pi (n),n}$ in Formula \eqref{eq:det1} corresponds to the
$n$-multicycle \begin{equation}\label{eq:sig} \sig=(\pi (1),1),
(\pi(2),2), \dots , {(\pi (n),n)}
\end{equation} in the digraph $\dGr _A$ of
$A$.  Conversely, any digraph $\dGr = (\gV, \gE)$ with $n$
vertices and edges labeled ~$1$ or $\1$ corresponds to the $n
\times n$ \textbf{adjacency matrix} $\adA(\dGr )$ over the
\semiring \ $\sbool$.

A matrix $A \in M_n(\sbool)$ is nonsingular iff there is exactly
one nonzero summand in \eqref{eq:det1} equals $1$, in particular
no summand is $\1$. This summand corresponds to a unique
$n$-multicycle of $\dGr _A$ with all edges labeled $1$ and
$\dGr_A$ has no other $n$-multicycle, otherwise the matrix $A$ is
singular.
\end{remark}

As in the case of determinants, the permanent of  a matrix  $A \in
M_n(\sbool)$ can written in terms of its minors. Denoting by
$A_{i,j}$ the minor of $A$ obtained by deleting the $i$'th row and
the $j$'th column, the permanent in Formula ~\eqref{eq:det1} can
be written equivalently as
\begin{equation}\label{eq:det2}
\per{A } := \sum _{j }a_{i,j} \per{A_{i,j}},\end{equation} for a
fixed $i = 1, \dots, n$.

It easy to verify that the permanent has the following properties:
\begin{enumerate}
   \item Permuting rows or columns of a superboolean matrix leave the permanent
   unchanged; \pSkip

   \item A matrix and its transpose have the same permanent;

     \pSkip
  \item Multiplication of any given row or
  column of a superboolean matrix  by  $\1$ or $0$ makes it
  singular.
\end{enumerate}

\begin{lemma}\label{lem:2.1.f}
A matrix $ A  \in M_n(\sbool)$ is nonsingular iff by independently permuting
columns and rows it has the triangular form
\begin{equation}\label{eq:trgform}
  A' :=  \(\begin{array}{cccc}
  1 & 0 &\cdots & 0 \\
  * & \ddots &  \ddots &\vdots\\
\vdots & \ddots & 1& 0\\
  * &  \cdots &*  & 1\\
   \end{array}\),
\end{equation}
 with all diagonal entries $1$, all entries above the diagonal are
 $0$, and the entries below the diagonal belong to  $\{1, \1, 0 \}$.

  Such
 reordering of $A$ is equivalent to multiplying the matrix $A$ by two
 permutation matrices $\Pi_1$ and~ $\Pi_2$ on the right and on the
 left, respectively, i.e., $A' := \Pi_1 A \Pi_2$.
\end{lemma}

\begin{proof}
Any matrix $A'$ of the triangular form \eqref{eq:trgform} is
nonsingular since the only permutation whose evaluation is not
equal to zero in \eqref{eq:det1} is the identity permutation,
which has value $1$ by construction, and therefore $\per{A'} = 1$.
\pSkip

$(\Rightarrow)$: Assume that $A = (a_{i,j})$ is nonsingular, then
Equation \eqref{eq:det1} has a unique summand (corresponding to
unique permutation $\pi_0 \in S_n$) of value $1$ and all other
summands (corresponding to permutations $\neq \pi_0$ in $S_n$) are
of value $0$. Permuting
 rows (columns) of $A$, we may assume that $\pi_0$
is the identity permutation, i.e., $a_{i,i} =1$ for any $i =
1,\dots,n$. Then, each of the  vertices of the digraph $\dGr _A$
of $A$ has a self loop $\sig_i = (i,i)$. Then, by Remark
\ref{rmk:perma}, since $A$ is nonsingular, $\dGr _A$ has the
unique $n$-multicycle $\sig$ consisting of these $n$ self-loops.

Let $\widetilde{\dGr} _A$ be the digraph obtained from $\dGr_A$ by
deleting all the self loops $\sig_i$. $\widetilde{\dGr}_A$ is then
an acyclic digraph, i.e., has nos cycles, since otherwise
$\widetilde{\dGr}_A$ would have a cycle which together with some
other self-loops $\sig_i$ of $\sig$ in $\dGr_A$ composes another
$n$-multicycle  (in $\dGr_A $) which would contradict the
nonsingularity of $A$, since then $\dGr_A$ would have more then
one $n$-multicycle, cf. Remark \ref{rmk:perma}. Thus, the digraph
$\widetilde{\dGr}_A$ can be reordered such that $i
> j $ for any edge $(i,j)$; in other words $a_{i,j} = 0$ for any
$j \geq i$. This reordering is equivalent to independently
permuting columns and rows of the associated matrix. Joining back
the self-loops $\sig_i$ that were omitted to the vertices of
$\widetilde{\dGr}_A$, corresponding to the diagonal entries of the
adjacency matrix $\adA(\widetilde{\dGr}_A)$, we get the desired
matrix $A' = \adA(\widetilde{\dGr}_A) + I$, which is of the Form
\eqref{eq:trgform}. \pSkip

$(\Leftarrow)$:
%
   This can  be seen directly since multiplying the matrix $A$
by a permutation matrix on left (resp. right) is equivalent to
permuting rows (resp. columns) of $A$. But, as known, permuting
rows or columns of a superboolean matrix leaves the permanent
   unchanged.
\end{proof}

%

Lat $A$ be an $m \times n$ superboolean matrix. We say that an $k
\times \ell$ matrix $B$, with $k \leq m$ and $\ell \leq n$,  is a
\textbf{submatrix} of~$A$ if $B$ can be obtained by deleting rows
and columns of~$A$. In particular, a \textbf{row} of a matrix $A$
is an $1 \times n$ submatrix of $A$, where a \textbf{subrow} of
$A$ is an $1 \times \ell$ submatrix of $A$, with  $\ell \leq n$. A
\textbf{minor} is a submatrix obtained by deleting exactly one row
and one column of a square matrix.

\begin{definition}\label{def:marker}
A \textbf{marker} $\rho$ in a matrix is a subrow having a single
$1$-entry and all whose other entries are $0$; the length of
$\rho$ is the number of its entries. A marker of length $k$ is
written $k$-marker.
\end{definition}
For example the nonsingular matrix $A'$ in \eqref{eq:trgform} has
a $k$-marker for each $k = 1,\dots, n$, appearing in this order
from bottom to top. (Note that in general markers need not be
disjoint.)

\begin{corollary}\label{cor:nonsing}
If a matrix $ A  \in M_n(\sbool)$ is a nonsingular matrix, then
$A$ has an
$n$-marker. 
\end{corollary}

\begin{proof} Since $A$ is nonsingular, by  Lemma \ref{lem:2.1.f},
it can be reordered to the From \eqref{eq:trgform}. Then it is
easy to see that the top row is an $n$-marker.
\end{proof}

Note that if $A$ is $n \times n$ nonsingular matrix then it has a
$k$-marker for any $k = 1, \dots, n$, and by Lemma~\ref{lem:2.1.f}
we have such (disjoint) markers with each lies in a different row.
On the other hand, a ghost matrix has no markers at all.

\begin{example} The following are all the possible  nonsingular $2 \times 2$ superboolean
matrices, up to reordering of columns and rows:
 $$\vMat{1}{0}{0}{1}, \quad \vMat{1}{1}{0}{1}, \quad \vMat{1}{\1}{0}{1},$$
each has a 2-marker.
\end{example}

We define the superboolean $n$-space $\sbool^{(n)} = \sbool \times
\cdots \times \sbool$ as the direct product of $n$ copies of
$\sbool$. The elements  of $\sbool^{(n)}$ are the $n$ tuples
$(a_1, \dots, a_n)$ with entries $a_i$ in $\sbool$, which we call
superboolean vectors. A vector $v$ in $\sbool^{(n)}$ is boolean if
all of whose entries are in $\tTzB$. A vector whose entries are
all in $\tGz$ is called a \textbf{ghost vector}.

\begin{definition}[{\cite[Definition 1.2]{IzhakianTropicalRank}}]\label{def:tropicDep}
A collection of vectors  $v_1,\dots,v_m \in \sbool^{(n)}$ is said
to be \textbf{dependent}
 if there exist $\al_1,\dots,\al_m \in \tTzB$,  not all of them $0$,
 for which
$$
   \al_1 v_1 +  \cdots + \al_m  v_m \in
  \tGz^{(n)}.
$$
Otherwise the vectors are said to be \textbf{independent}.
\end{definition}
Note that when one of the $v_i$'s is ghost, or $v_i = v_j$ for
some $i \neq j$, then the vectors are dependent. A set of nonzero
boolean vectors can also be dependent; for example the vectors
$$ (0,1,1), \quad (1,0,1), \quad (1,1,0) $$
are dependent, since their sum is $(\1,\1,\1)$. (This example also
shows that the notions of dependence and spanning do not coincide
in this framework, since none of these vectors can be written in
terms of the others.)

The following results can be found in
\cite{IzhakianRowen2008Matrices} and
\cite{IzhakianRowen2009TropicalRank} for the general supertropical
setting, however, to make this paper self-contained we bring the
superboolean versions of these results with easier proofs.

%
%
%
%
%
%
%

\begin{definition}\label{rdef}
A set $v_1, \dots, v_k$ of vectors has
\textbf{rank defect} $\ell$ if there are $\ell$ columns, denoted
$j_1, \dots, j_\ell$, such that $v_{i,j_u} = 0$ for all $1 \le i
\le k$ and $1\le u \le \ell$.
\end{definition}

For example, the vectors $v_1 = (1,0,1,0), v_2 = (0,0,0,1), v_3 =
(1,0,0,0)$ have rank defect $1$, since they are all $0$ in the
second column; $v_1$ and $v_2$ have rank defect $2$.

\begin{proposition}[{\cite[Proposition
2.10]{IzhakianRowen2008Matrices}}]\label{ssing} An $n\times n$
matrix $A$ has permanent $0$, iff, for some $1 \le k\le n$, $A$
has $k$ rows having rank defect $n+1-k.$
\end{proposition}
\begin{proof} $(\Leftarrow)$ The case of $k=n$ is obvious, since
some column is entirely $0$. If $n>k$, we take one of the columns
$j$ other than $j_1, \dots, j_k$ of Definition~\ref{rdef}. Then
for each $i$, the minor $A_{i,j}$ has
  at least $k-1$ rows with rank defect
$(n-1)+1-k$, so has permanent $0$ by induction; hence $\per{ A } =
0,$ by Formula~\eqref{eq:det2}. \pSkip

$(\Rightarrow)$ We are done if all entries of $A$ are $0,$ so
assume for convenience that $a_{n,n}\ne 0$. Then, $\per{A_{n,n}} =
0 $ and, by induction, $A_{n,n}$ has $k\ge 1$ rows of rank defect
$$(n-1)+1-k = n-k.$$

For notational convenience, we assume that $a_{i,j} = 0$ with  $1
\le i \le k$ and $1 \le j \le n-k.$ Thus,  $A$ has the partition
 $$ A = \(
\begin{array}{c|c}
 (0) & B' \\[1mm] \hline
B''   & \overset{ } {C }
\end{array} \) ,$$
where $(0)$ stands for the $k \times (n\! -\! k)$ zero matrix,
$B'$ is a $k\times k$ matrix, $B''$ is an $(n\!-\! k )\times (n\!
-\! k)$ matrix, and $C$ is an $(n\!-\! k) \times k$ matrix.

By inspection, $\per{B' }\per{B'' } = \per{A } =0$; hence $\per{B'
}= 0$ or $\per{B'' } = 0$. If $\per{B' } = 0$,
 then, by induction, $B'$ has $k'$ rows of rank defect
$k+1-k',$ so altogether, the same $k'$ rows in $A$ have rank
defect $(n-k) + k+1-k' = n+1-k',$ and we are done by taking $k'$
instead of $k$.

If $\per{B'' } = 0$,
 then, by induction, $B''$ has $k''$ rows of rank defect
$(n-k)+1 -k'',$ so altogether, these $k+k''$ rows in $A$ have rank
defect $n+1-(k+k''),$ and we are done, taking $k+k''$ instead of
$k$.
 \end{proof}

\begin{example}\label{exm:1}
Suppose the rows of $A \in M_2(\sbool)$, $A = (a_{i,j})$, are
dependent. Then there are $\al_1, \al_2 \in \tGz$ such that
$\al_1(a_{1,1},a_{1,2}) + \al_2 (a_{2,1},a_{2,2}) \in \tGz^{(2)}$.
If $\al_1 = 0$, then $\al_2 = 1$ and $(a_{2,1},a_{2,2}) \in
\tGz^{(2)}$, implying $A$ is singular. By the same argument if
$\al_2 =1 $ then $A$ is again singular.

Assume that $\al_1 = \al_2 =1 $, then $a_{1,1} + a_{2,1} \in \tGz$
and $a_{1,2} + a_{2,2} \in \tGz$, which implies that $\per{A} =
a_{1,1} a_{2,2} +  a_{1,2} a_{2,1} \in \tGz$, i.e., $ A$ is
singular.
\end{example}

\begin{lemma}\label{lem:regularityToIndependent}
The  rows of any singular $n \times n$  matrix are dependent.
\end{lemma}

\begin{proof}
We induct on $n$,   the case   $n =1$ is obvious. (The case $n=2$
is provided in Example \ref{exm:1}.) Permuting independently rows
and columns, we may assume that the value of \eqref{eq:det1}, up
to $\nu$-equivalence, is the attained by identity, i.e.,
$$ \per{A}  \nucong a_{1,1} \cdots a_{n,n}.$$
Let $v_1, \dots, v_n$ denote  the rows of $A$. \pSkip

\emph{Case I:} Assume that $\per{A} =  \1 $. Let $A'$ be an $m
\times m $ singular submatrix of $A$ with $\per{A'} \nucong 1$
whose diagonal lies on the diagonal of $A$.  For notational
convenience, we assume that such a singular submatrix $A'$ with
$m$ minimal is the upper left submatrix of $A$; in particular if
$a_{i,i} = \1$, for some $i$, renumbering the indices we may
assume that $a_{1,1} = \1$.

Let
\begin{equation}\label{eq:al}
\al_i = \left\{
\begin{array}{lll}
  0 & \text {if } & \per{A_{i,1}} = 0, \\[1mm]
  1 & \text {if } &  \per{A_{i,1}} \nucong 1.
  \end{array}
  \right.
\end{equation}
By assumption, some $\al_i = 1$. We claim that $\sum _i \al_i v_i
\in \tGz^{(n)}$, that is,
\begin{equation}\label{eq:toProv} \sum_i \al_i a_{i,j} \in
\tGz, \quad \text{ for each $j = 1,\dots,n$}.
\end{equation}

When $j=1$, Formula  \eqref{eq:toProv} is just the expansion of
$\per{A}$, up to $\nu$-equivalence, along the first column of~$A$,
which we claim is $\1$. Indeed, when $m =1$, i.e., $a_{1,1} = \1$,
 we are done since $\per{A} \nucong a_{1,1} \per{A_{1,1}}$.
Otherwise, since $m > 1 $ is minimal, there is some other
permutation besides the identity that also attains $\per{A'};$
that is $a_{1,1} \per{A'_{1,1}} \nucong a_{1,i} \per{A'_{1,i}}
\nucong 1$ for some $1 < i \leq m $. Thus,  $$a_{1,1}
\per{A'_{1,1}} a_{m+1,m+1} \cdots a_{n,n} \nucong a_{1,i}
\per{A'_{1,i}}a_{m+1,m+1} \cdots a_{n,n},$$  and therefore  $\al_1
a_{1,1} \nucong \al_i a_{i,1}$.

Suppose  $j > 1$, if  $ \sum _i  \al_{i} a_{i,j} = 0$ we are done.
So, assume that $\al_\ell a_{\ell,j} \neq 0$ for some $\ell$. Then
$a_{\ell,j} \nucong 1$ and
\begin{equation*}\label{eq:s1}\al_\ell =  \per {A_{\ell,1}} \nucong \prod_{i \neq \ell} a_{i,\sig(i)}
\nucong 1 ,
\end{equation*}
for some $\sig \in S_n$ with  $\sig(\ell) = 1$. Let  $u $ be the
index for which $\sig(u) = j$; in particular $u \neq \ell$. Let
$\sig' \in S_n$ be the  permutation with $\sig'(u) =1$,  $\sig'
(\ell) = j$, and $\sig'(i) = \sig(i)$ for each $i \neq u, \ell$.
Then we have
\begin{equation*}\label{eq:s1}\al_u =  \per {A_{u,1}}
\nucong \prod_{i \neq u } a_{i,\sig'(i)}
\nucong 1 .
\end{equation*}
Thus, $\al_u a_{u,j}$ and $\al_\ell a_{\ell,j} $ are two different
summands  in Formula \eqref{eq:toProv} with  $\al_u a_{u,j}
\nucong \al_\ell a_{\ell,j} \nucong 1$, as desired. \pSkip

\emph{Case II:} Suppose that  $\per{A } = 0$ and $A$ has a minor
$A_{i,j}$ with $\per{A_{i,j}} \neq 0$. Permuting independently
rows and columns we may assume that $i = j =1$. We define the
$\al_i$'s as in \eqref{eq:al} and claim that Equation~
\eqref{eq:toProv} is true for these $\al_i$'s. When $j=1$,
Formula~ \eqref{eq:toProv} is just the expansion of $\per{A}$
along the first column of~$A$, which we know is $0$ since $\per{A}
= 0$. For $j
> 1$ we apply the same argument as in Case I.

\pSkip

\emph{Case III:} Assume that $\per{A } = 0$ with all
$\per{A_{i,j}}$ are $0$.  We take $m$ maximal such that $A'$ is an
$m\times m$ submatrix with a minor of permanent  $\neq 0$. By
induction, we may assume that $m= n-1.$ Furthermore, it is enough
to find a dependence among the $k$ rows obtained in
Proposition~\ref{ssing}, so, again, by induction, we may assume
that $k = n-1$, and the entries in the first column are all $0$.
Since $a_{1,1} = 0$ and $\per{A'} \neq 0$, the proof is then
completed by the argument of
 Case II.
\end{proof}

\begin{theorem}[{\cite[Theorem
2.10]{IzhakianTropicalRank}}]\label{thm:regularityToIndependent}
The rows (columns) of a matrix $A \in M_n(\sbool)$ are independent
iff $A$ is
 nonsingular, i.e., $\per{A} =1$.
\end{theorem}

\begin{proof}  $(\Rightarrow):$ By Lemma
\ref{lem:regularityToIndependent}.

\pSkip  \noindent  ${ (\Leftarrow)}:$ Suppose $A$ is nonsingular
and assume by contradiction that the rows $v_1, \dots ,v_n$ of $A$
are dependent. Permuting independently rows and columns do not
change the dependence relations of $v_1, \dots ,v_n$, so by Lemma
\ref{lem:2.1.f} we may assume that $A$ is of the Form
\eqref{eq:trgform} and there are $\al_1, \dots, \al_n \in \tGz $,
not all of them~ $0$, such that $\sum_i \al_i v_i \in \tGz^{(n)}$.
Let $\sum_i \al_i v_i = w$, where $w = (w_1, \dots, w_n)$. Suppose
that $i$ is the largest index for which $\al_i =1$, then it easy
to see that $w_i = 1$ -- a contradiction. Thus, the rows of $A$
are independent.
\end{proof}

\begin{corollary}[{\cite[Corollary 2.13]{IzhakianRowen2009TropicalRank}}]\label{cor:n+1vectors} Any $k
> n$ vectors in $\sbool^{(n)}$ are  dependent.\end{corollary}

\begin{proof} Assume $v_1, \dots, v_{n+1}$ are
vectors in $\sbool^{(n)}$ and consider the $(n+1) \times n$ matrix
whose rows are these vectors. Extend this matrix by duplicating
one of the columns to get a singular matrix,  whose rows are
dependent by Theorem \ref{thm:regularityToIndependent}.
\end{proof}

\begin{theorem}[{\cite[Theorem
3.6]{IzhakianRowen2009TropicalRank}}]\label{thm:base} Let $A =
(a_{i,j})$ be an $m \times n $ matrix with $n \geq m$, and suppose
that each of whose $m \times m$ submatrices is singular. Then the
rows $v_1, \dots, v_m$ of $A$ are dependent.
\end{theorem}

\begin{proof}  We
induct on $n$, having proved the theorem for $m=n$ in Theorem
\ref{thm:regularityToIndependent}. Thus, we may assume that $m<n$.

For each   $j = 1,\dots,  n$  we define $v_i^{(j)}$ to be the
vector obtained by deleting the $j$ entry, and $A^{(j)}$ to be the
submatrix of $A$ obtained by deleting the $j$ column of $A.$
Namely, the vectors $v_1^{(j)}, \dots, v_m^{(j)}$ are the rows of
$A^{(j)}$ and by induction are dependent, i.e.,  there are
$\al_{i,j} \in \tTzB$ such that $\sum_{i=1}^n \al_{i,j}v_i^{(j)}
\in \tGz^{(n-1)} .$
 We are done if $\sum _i \al_{i,j}a_{i,j} \in
\tGz$ for some $j$, since then $\sum _i \al_{i,j}v_i \in
\tGz^{(n)}$.  So, we may assume for each $j$ that $\sum _i
\al_{i,j}a_{i,j} =1 $. Pick $i_j$ such that $\sum _i
\al_{i,j}a_{i,j} = \al_{i_j,j}a_{i_j,j} =1.$

Since there are at least $m+1$ values of $i_j,$ and by pigeonhole
principle two are the same, say $i_{j'} = i_{j''}$. To ease
notation, we assume that $i_{j'} = i_{j''} = 1$. Thus,
$\al_{1,j'}a_{1,j'} = 1$  and $\al_{1,j''}a_{1,j''} = 1$, and in
particular $\al_{1,j'} = 1$ and $\al_{1,j''} =1$.  Let
%
$$ \al_i =
\begin{cases} \al_{1,j''}\al_{i,j'}& \quad \text{if}\quad \al_{1,j''}\al_{i,j'}
= \al_{1,j'}\al_{i,j''},\\
\al_{1,j''}\al_{i,j'} + \al_{1,j'}\al_{i,j''} & \quad
\text{else}.\end{cases}$$

 We need to show that for each $j$
\begin{equation}\label{check1}\sum_i \al_i a_{i,j}  \in \tGz.\end{equation} The
case of $j \ne j', j''$ is immediate, since we are given $\sum_i
\al_{i,j'} a_{i,j} \in \tGz$ and $\sum_i \al_{i,j''} a_{i,j} \in
\tGz$, implying at once that $\sum_i \al_i a_{i,j} \in \tGz$.
Thus, we need to verify  \eqref{check1} for $j = j'$ and $j =
j''$; by symmetry, we assume that $j = j'$. By assumption,
$\al_{i,j'} a_{i,j'} = 0$ for each $i \neq 1.$ Thus, $\al_{1 }
a_{1,j'}= \al_{1,j'}\al_{1,j''} a_{1,j'}  > \al_{1,j''}\al_{i,j'}
a_{i,j'} = 0$. On the other hand, $$\al_{1,j'}\al_{1,j''} a_{1,j'}
\ds \leq \sum _i \al_{1,j'}\al_{i,j''} a_{i,j'} \ds =
\al_{1,j'}\sum _i \al_{i,j''} a_{i,j'} \in \tGz,$$ by the
dependence of $v_1^{(j'')}, \dots, v_m^{(j'')};$ so we conclude
that
$$\sum_i \al_i a_{i,j}   =  \al_{1,j'}\sum _i
\al_{i,j''} a_{i,j'} \in \tGz,$$ as desired.
\end{proof}

\begin{corollary}\label{cor:sinDep}
The columns (resp. rows) of an $m \times n$  matrix $A$, with $n
\leq m$ (resp. $n \geq m$), are independent iff $A$ contains an $n
\times n$ (resp. $m \times m$) nonsingular submatrix.
\end{corollary}

\begin{proof} ${(\Rightarrow)}:$ If all the  $n \times n$
submatrices of $A$ are singular then the columns of $A$ are
dependent by Theorem~\ref{thm:base}.

\pSkip ${(\Leftarrow)}:$ Let $A'$ be an $n\times n$ nonsingular
submatrix of $A$, then its columns are independent  by Theorem~
\ref{thm:regularityToIndependent}. Since the columns of $A'$ are
subcolumns of $A$, then the columns of $A$ are also independent.
\end{proof}

The \textbf{column rank} of a superboolean matrix $A$ is defined
to be the maximal number of independent columns of $A$. The
\textbf{row rank} is defined similarly with respect to the rows of
$A$.

We denote the rank of a superboolean matrix $A$ by
$\rnk_\sbool(A)$, or simply by $\rnk(A)$, when it is clear form
the context.  Note that an $n \times n$ nonzero  matrix has rank
$0$ if all of its entries are in $\tGz$, i.e., when $A$ is a ghost
matrix.

\begin{corollary}[{\cite[Corollary 3.7]{IzhakianTropicalRank}}]\label{cor:nRank}
A matrix in $M_n(\sbool)$ is of rank $n$ iff it is nonsingular.
\end{corollary}
\begin{proof} Immediate by \ref{thm:regularityToIndependent}.
\end{proof}

The rank of a superboolean matrix is then invariant under the
following operations:
\begin{enumerate}\eroman

   \item permuting of rows (columns); \pSkip

   \item deletion of a row (column) whose entries are all in $\tGz$; \pSkip

   \item deletion of a repeated  row or column.

\end{enumerate}

\begin{theorem}[{\cite[Theorem
3.11]{IzhakianTropicalRank}}]\label{thm:rnkSing} For any matrix
$A$ the row rank and the column rank are the same, and this rank
is equal to the size of the maximal nonsingular submatrix of $A$.
\end{theorem}

\begin{proof}
 Let $k$ be the
 row rank of $A$, and let $\ell$ be the rank of the maximal nonsingular matrix.
 Clearly $k\ge \ell,$ since any $\ell\times \ell$ nonsingular matrix has
independent rows by Theorem \ref{thm:regularityToIndependent}. On
the other hand, Theorem~\ref{thm:base} shows that $k\le \ell,$ so
$k = \ell$. The assertion for columns follows by considering the
transpose matrix, since obviously the submatrix rank of a matrix
and of its transpose are the same,   both being equal to the size
of a maximal nonsingular square submatrix.
\end{proof}

\begin{corollary}\label{cor:nRankTran}
The rank of a superboolean matrix is invariant under
transposition.
\end{corollary}


\subsection{Ranks of matrices}
A boolean matrix $A \in M_n(\bool)$ can be formally considered as
a matrix over a field $\bF$, i.e., a member of the ring of
matrices $M_n(\bF)$, where $1$ and $0$ are respectively the
multiplicative unit and the zero of $\bF$.   In this view, the
\textbf{field rank} of $A$ is defined to be the standard matrix
rank of $A$ in $M_n(\bF)$; this rank is denoted by $\rnk_\bF(A)$.

\begin{proposition}\label{prop:rank}
$\rnk_\bF(A) \geq \rnk_\sbool(A)$ for any $A \in M_n(\bool)$ and
over any  field $\bF$.
\end{proposition}

\begin{proof}
Suppose $\rnk_\sbool(A)  = k$, where  $ 0 \leq k \leq n$. If $k=0$
we are done, since $A$ is a boolean  matrix and thus $A = (0)$.
Otherwise, by Corollary \ref{cor:sinDep}, $A$ has a $ k\times k$
nonsingular submatrix $B$, which by Lemma \ref{lem:2.1.f} can be
permuted to the triangular form \eqref{eq:trgform}, for which we
clearly have $\rnk_\bF(B) = k$. Therefore, $\rnk_\bF(A) \geq k$.
\end{proof}

\begin{example}  Consider the matrix
$$A = \vvMat{1}{1}{0} {0}{1}{1} {1}{0}{1}.$$
If $\bF$ is a field of characteristic $\neq 2$, then $\rnk_\bF(A)
= 3$, while $\rnk_\sbool(A) = 2$.
\end{example}

\section{Representations of hereditary collections}

\subsection{Classical representations of matroids over
fields}\label{sec:21.}

The traditional approach to represent a matroid  uses matrices
defined over
 fields, often finite fields, which in their turn generate vector
matroids as  explained below. This will be generalized later in
\S\ref{ssec:veHe} to hereditary collection with respect to the
superboolean semiring.

 In the sequel, we write $\kA$ to indicate
that a given matrix $A$ is considered as a matrix over the ground
structure $\Fld$ -- either a field or a \semiring.

\begin{notation}\label{nott}
Given a matrix $\kA$ and a subset $X \subseteq \Cl(\kA)$ of
columns  of $\kA$, we write $\cl{\kA}{X}$ for the submatrix of
$\kA$ having the columns $X$. Sometimes we refer to $\Cl(\kA)$ as
a collection of vectors, but no confusion should arise. Given also
a subset $Y \subseteq \Rw(\kA)$ of rows of $\kA$, we define
$\rwcl{\kA}{Y}{X}$ to be the submatrix of $\kA$ having the
intersection of columns $X$ and the rows $Y$, often also referred
to as a collection of sub-vectors.
\end{notation}

Any $m \times n$ matrix $\fA$ over a field $\bF$ gives rise to a
matroid $\M(\fA)$ constructed in the following classical way
\cite{whitney}. We label uniquely the columns of $\fA$ (realized
as vectors in $\bF^{(m)}$)  by a set $E := E(\fA)$, $|E| =
|\Cl(\fA)|$. The independent subsets $\tH := \tH(\fA)$ of
$\M(\fA)$ are  subsets of $E$   corresponding to column subsets of
$\fA$ that are linearly independent in~$\bF^{(m)}$. This
construction is well known, cf. \cite{oxley:matroid}, and
$\M(\fA):= (E(\fA),\tH(\fA))$ is called a \textbf{vector matroid}.

\begin{proposition}\label{2.1:prop}
$\M(\fA)$ is a matroid.
\end{proposition}

An equivalent way to describe the independent subsets of a vector
matroid $\M(\fA)$, using Notation~\ref{nott}, is as follows (WT
stands for ``witness''):
 \boxtext{
$\begin{array}{llll}
 \text{WT:} & X \in  \tH(\kA) & \iff & \exists Y \subseteq
\Rw(\kA) \text{ with } |X| = |Y|    \\[1mm]
& & & \text{ such that $\clrw{\kA}{X}{Y}$  is nonsingular over
$\bK$,}
\end{array}
$} where here we take $\bK = \bF$ to be a field. (This condition
is central in our development, to be used later for semirings as
well.)

A matroid  $\M'$ that is isomorphic (cf. Definition
\ref{def:HCIso}) to a vector matroid $\M(\fA)$ for some matrix~
$\fA$ over a field $\bF$, is said to be
\textbf{field-representable}, written
$\bF$\textbf{-representable}, and the matrix $\fA$ is called a
\textbf{field-representation}, written
$\bF$\textbf{-representation}, of $\M'$. We write $\fA(\M')$ for
an $\bF$-representation of~$\M'$, which need not be unique. Given
a subset $X \in \tH(\M')$, with $|X| = k$, a nonsingular $k \times
k$ minor $\clrw{\fA}{X}{Y}$, $Y \subseteq \Rw(\fA)$, of the
$\bF$-representation $\fA := \fA(\M')$ of $\M'$ is termed a
\textbf{witness} of $X$ in $\fA$. (In particular the columns of
$\clrw{\fA}{X}{Y}$, and thus the columns of $\cl{\fA}{X}$, are
independent.)

The \emph{new simple idea of this paper} is to replace the role of
the field $\bF$, used for classical  matroid reorientations, by
some commutative semiring; this allows the representation of any
matroid, and moreover of any hereditary collection, as will be
described next. In this paper we take this commutative semiring to
be the superboolean \semiring \ and show that for some cases the
use of the boolean \semiring \ is sufficient.

\subsection{$\sbool$-vector
hereditary collection}\label{ssec:veHe} Given a matrix $\sbA$ over
the superboolean \semiring \ $\sbool$, by the same construction as
explained above for vector matroids, using condition WT  with $\bK
= \sbool$, we define the hereditary collection $\H(\sbA)$, where
now dependence of columns and nonsingularity of submatrices are
taken in the superboolean sense, cf. Definition
\ref{def:tropicDep}. Formally, we have the following {important
key  definition}:
\begin{kdefinition}\label{defn:VecHC} Given an $m \times n$  superboolean matrix
$\sbA$, we define  $\H(\sbA) := (E, \tH)$ to be the  hereditary
collection  with $E := E(\Cl(\sbA))$ corresponds uniquely to the
columns of $\sbA$, i.e., $|E| =|\Cl(\sbA)|$, and whose independent
subsets $\tH := \tH(\sbA)$ are column subsets of $\sbA$ that are
independent in the $m$-space $\sbool^{(m)}$, namely, satisfying
condition WT above for $\bK= \sbool$.

 We call $\H(\sbA)$  an
\textbf{$\sbool$-vector hereditary collection}, and say that it is
a $\bool$-vector hereditary collection when $\sbA$ is a boolean
matrix.
\end{kdefinition}

Having this notion of $\sbool$-vector hereditary collections, we
say that a hereditary collection $\H'$ is
\textbf{superboolean-representable}, written
${\sbool}$\textbf{-representable}, if it is isomorphic
(cf.~Definition~\ref{def:HCIso}) to an $\sbool$-vector hereditary
collection $\H(\sbA)$ for some superboolean matrix $\sbA$ and
write $\sbA(\H)$ for an ${\sbool}$\textbf{-representation} of
$\H$. When the matrix $\sbA(\H)$ is a boolean matrix, i.e., with
$1$, $0$ entries, we call this representation a \textbf{boolean
representation}, written ${\bool}$\textbf{-representation}, and
say that $\H$ is ${\bool}$\textbf{-representable}. We use the same
terminology as before and called a $k \times k$ nonsingular minor
$\clrw{\sbA}{X}{Y}$ of $\sbA$ the \textbf{witness} of the
independent subset  $X \subseteq E$, $|X|=k$, in the
$\sbool$-representation $\sbA := \sbA(\H)$ of the hereditary
collection $\H = (E, \tH)$.

\begin{remark} Given a superboolean matrix $\sbA$ the $\sbool$-vector hereditary collection
$\H(\sbA)$  needs not be a matroid. For example the vector
hereditary collection $\H(\sbA)$ of the matrix
$$ \sbA = \( \begin{array}{cccc}
 1 & 0 & \1& 1\\
 0 & 1  & 1 & \1\\
   \end{array}\) $$
is not a matroid. The independent subsets of $E = \{1, 2,3, 4 \}$
are $\{1,2 \}$, $\{1,3 \}$, $\{2,4 \}$,  all the singletons of
$E$, and the empty set.  Therefore, Axiom MT is not satisfied for
the subset $\{4\}$ with respect to $\{1,3\}$.

When a matrix $\blA$ is a nonzero boolean matrix, yet the vector
hereditary collection $\H(\blA)$ needs not be a matroid. For
example consider the $\sbool$-vector hereditary collection
$\H(\blA)$ of the matrix
$$ \blA = \( \begin{array}{cccc}
 1 & 1 & 0 & 1 \\
 0 & 1 & 1 & 1 \\
 0 & 0 & 1 & 1 \\
   \end{array}\) $$
whose bases are $\{ 1,2,3 \}$, $\{1,2,4\}$, and $\{ 3,4 \}$. Thus,
$\H(\blA)$ is not a matroid.
\end{remark}

\begin{example}\label{exp:U2n} The uniform matroid $U_{2,n}$ (cf. Example \ref{exp:1.2}.(c)) is
$\bool$-representable by the $(n-1) \times n$ boolean  matrix
$$  \blA(U_{2,n}) = \(\begin{array}{ccccc}
         0 & 1 &   \cdots & & 1 \\
                  1 & 0 &  1 & \cdots & 1 \\
               \vdots &  \ddots & \ddots & \ddots & \vdots \\
              1 & \cdots & 1& 0 & 1 \\
            \end{array}
 \).
$$
One sees that any pair of columns of $\blA(U_{2,n})$ are
independent since they contain either one $0$-entry or two
$0$-entries in different positions,  and thus a $2 \times 2 $
witness of the form $\vMat{1}{0}{0}{1}$ or   $\vMat{1}{1}{0}{1},$
respectively.  On the other hand, any column has at most one
$0$-entry, and therefore any $3 \times 3$ submatrix is singular.
Thus, any subset of more than $2$ columns is dependent.
\end{example}




\subsection{Superboolean  representations of  hereditary collections}
We are now ready for one of our main theorems of this paper.

\begin{theorem}\label{thm:hdCol}
Every hereditary collection $\H = (E, \tH)$ over a ground set $E$
of $n$ elements is $\sbool$-representable by an $m \times n$
superboolean matrix.
\end{theorem}
\begin{proof} We prove the theorem by  constructing an explicit
$\sbool$-representation $\sbA(\H)$ for a given hereditary
collection $\H$. The columns of $\sbA(\H)$ will be labeled by the
ground set $E$ and each independent subset $X \subseteq E$, with
$|X| = k$, will correspond to a column subset labeled by $X$ and
containing a witness, i.e., a ~$k\times k$ nonsingular minor, cf.
WT above.


When $E$ is empty, then $\H$ is represented by the formal $0
\times 0$ matrix, i.e., by the empty matrix. So, throughout we
assume that $|E| > 0$.
 In the case that $\tH = \{ \emptyset \}$, $\H$ can be
$\sbool$-represented by any  $m \times n$ ghost matrix, and in
particular by any $1 \times n$ ghost matrix.

Suppose that $\tH$ contains a nonempty subset of $E$, and let $
\tB(\H) = \{J_1, \dots, J_\ell \}$ be the set of bases of $\H$.
Given a basis $J_i \in \tB(\H)$, $J_i = \{b_{i_1}, b_{i_2}, \dots,
b_{i_{m_i}}\} $, with $m_i$ elements, we define the $m_i \times n
$ matrix $\sbC^{(i)} := \sbA(J_i)$ having the $ m_i \times m_i$
nonsingular minor, whose columns correspond to the elements
$b_{i_1}, b_{i_2}, \dots, b_{i_{m_i}}$ of $J_i$ assuming $b_{i_1}
\leq  {i_2} \leq  \cdots  \leq {i_{m_i}}$, to be of the form (we
use Notation \ref{nott} for submatrices)
\begin{equation}\label{eq:setRep}
 \cl{\sbC^{(i)}}{J_i} = \(\begin{array}{cccc}
        1& 0 &  \cdots & 0 \\
               \1 &  \ddots & \ddots & \vdots \\
              \vdots &  \ddots & 1 & 0 \\
              \1 & \cdots & \1 & 1 \\
            \end{array}
 \),
\end{equation}
i.e., $\cl{\sbC^{(i)}}{J_i}$ has $1$ on the main diagonal, $0$
strictly above the diagonal, whence  $\1$ strictly below the
diagonal;
 all the other entries of $\sbC^{(i)}$ are $\1$. Namely,
after permuting the columns of $\sbC^{(i)}$, we have the form
\begin{equation}\label{eq:setRep2}
 \sbA  (J_i) = \sbC^{(i)} := \(\begin{array}{ccc|ccc}
 &    & & \1 & \cdots & \1 \\
 & \cl{\sbC^{(i)}}{J_i}   & & \vdots & \ddots & \vdots \\
 &    & & \1 & \cdots & \1 \\
            \end{array}
 \),
\end{equation}
with $i_1 = 1, i_2 = 2, \dots, i_{m_i} =m_i$.  Clearly
$\sbC^{(i)}$ has rank $m_i$, since it contains the $m_i \times
m_i$ witness $\cl{\sbC^{(i)}}{J_i}$, whence its rows are linearly
independent by Corollary \ref{cor:sinDep}. Accordingly,
$\sbA(J_i)$ is an $\sbool$-representation of the hereditary
collection $\H_i = (E, \Pow(J_i))$, $J_i \subseteq E$.

Having the matrices $\sbC^{(i)} = \sbA(J_i)$ at hand, for each
basis $J_i$ of $\H$, $i = 1, \dots, \ell$, we construct the matrix
\begin{equation}\label{eq:setRep1}
\sbB := \(
\begin{array}{l c}
\sbA(J_\ell) := & \begin{array}{|ll|ll|ll|}\hline
      &  &   &            & &  \\
    & (\1) \ & \cl{\sbC^{(\ell)}}{J_\ell} & &     (\1)     &  \\
    &  &  &             & &  \\  \hline
              \end{array} \\
&               \vdots \qquad \qquad \qquad \vdots \\
\sbA(J_2) := & \begin{array}{|lll|lll|}\hline
      &  &   &            & &  \\
    & (\1) \qquad & &          & \cl{\sbC^{(2)}}{J_2} & \\
    &  &  &             & &  \\  \hline
              \end{array} \\
\sbA(J_1) := & \begin{array}{|lll|lll|}\hline
      &  &   &            & &  \\
    & \cl{\sbC^{(2)}}{J_1} & & \qquad  (\1)         &  & \\
    &  &  &             & &  \\  \hline
              \end{array} \\
               \end{array} \)
\end{equation}
by \emph{stacking} the matrices $\sbA(J_i) $ one over the other
with respect to their columns labeling; $(\1)$ stands for a matrix
all of its entries are $\1$. (Note that Form \eqref{eq:setRep1}
only illustrates the construction of $\sbB$, the columns of
$\cl{\sbC^{(i)}}{J_i}$ need not be consecutive in $\sbA(J_i)$, and
these blocks need not overlap each other.)

The matrix $\sbB$  has the following properties:
\begin{enumerate} \ealph
    \item every  row has exactly one $1$-entry; \pSkip

    \item each $k$-marker, cf. Definition  \ref{def:marker}, is
    contained in a $k \times k$ witness which is a $k \times k$ submatrix of
    $\sbA(J_i)$, for some $i$.
\end{enumerate}


To prove that $\sbB$ is a proper  $\sbool$-representation of $\H$,
we need to
    verify that the  dependence and independence relations satisfied   by the columns of $\sbB$
    are exactly those of $\H$. These relations
  have been recorded separately by the matrices $\blA(J_i)$ for the
  bases $J_i$
of $\H$.

Let $Y_i
    \subseteq \Rw(\sbB)$  denote the subset of rows of $\sbB$
    corresponding to the matrix $\blA(J_i)$. Given a $k$-subset $X \subseteq E $, i.e., $|X| =
    k$, we claim
    that $\Cl(\cl{\sbB}{X})$ are independent iff $\Cl(\clrw{\sbB}{X}{ Y_i})$
    are
    independent for some $i = 1,\dots, \ell$.
\pSkip
     $(\Leftarrow):$  If $\Cl(\clrw{\sbB}{X}{Y_i})$  are
    independent, then $\clrw{\sbB}{X}{Y_i}$  contains a $ k \times k$
    witness, where $k \leq m_i$, which is also contained in $\cl{\sbB}{X}$, and we are
    done by Corollary \ref{cor:sinDep}.
\pSkip
 $(\Rightarrow):$ Suppose  $\Cl(\cl{\sbB}{X})$ are
independent then it contains a $k \times k$  witness, which by
Corollary~\ref{cor:nonsing} contains a $k$-marker $\rho$.  But, by
construction,  $\rho$ belongs to a $k \times k$ witness that is
contained in some $Y_i \subseteq \Rw(\blA(J_i))$; therefore
$\Cl(\clrw{\sbB}{X}{Y_i})$ is independent again  by Corollary
\ref{cor:sinDep}.
\end{proof}

\begin{corollary}
Any matroid is $\sbool$-representable.
\end{corollary}

We aim for an ${\sbool}$-representation of minimal size, i.e., has
a minimal number of rows. Let us start with a naive upper bound
obtained directly from  our construction in the proof of Theorem
\ref{thm:hdCol}, that is
\begin{equation}\label{eq:upperBond} m \leq \sum_{J
\in \tB(\H)} |J|,
\end{equation} where we recall that $|J|$
stands for the cardinality of a basis $J$ of $\H$. (Clearly, by
Corollary \ref{cor:n+1vectors}, the lower bound is determined by
the rank of the hereditary collection $\H$.)

\begin{remark} Given an $\sbool$-representation  $\blA(\H)$ of a hereditary
collection $\H$, one can reduce $\blA(\H)$ by erasing repeated
rows, leaving a single representative for each subset of identical
rows; ghost rows can also be  omitted, as long  as $\blA(\H)$
remains with at least one row.
\end{remark}

\begin{example}
Let $E = \{1,2,3 \}$, and let $\tH = \{ \emptyset, \{1\}, \{2\},
\{3\}, \{1,2 \}, \{1,3\} \}$ be the independent subsets of the
hereditary collection $\H = (E, \tH)$. Thus, $ \{1,2 \}, \{1,3\}$
are the bases of $\H$.  Using the construction of the proof of
Theorem \ref{thm:hdCol} we obtain the $\sbool$-representation
$$ \sbA(\H)
 = \(
\begin{array}{ccc}
            1 & \1 & 1 \\
               0 & \1 & 1 \\
               1 & 1 & \1 \\
               0 & 1 & \1 \\
             \end{array}\).
$$
However, this  hereditary collection $\H$ can also be represented
by the smaller matrix
$$ \sbA'(\H) = \( \begin{array}{ccc}
             1 & 0 & \1 \\
                  0 & 1 & 1 \\
                \end{array}\).$$
                This is this a minimal possible $\sbool$-representation of
                $\H$, i.e., $m =1$.
\end{example}

\begin{example}\label{ex:Umn} The uniform matroid $U_{m-1,m}$ can be $\sbool$-represented
by the $m \times m$ matrix $\sbA = (a_{i,j})$ of the form:
\begin{equation}\label{form:2}
\sbA = \( \begin{array}{ccccc}
           1 & \1 & 0 & \cdots&  0 \\
             0 & 1 & \ddots & \ddots &\vdots \\
             \vdots & \ddots &  & \ddots & 0 \\
             0 &  & \ddots & 1 & \1 \\
             \1 & 0 & \cdots & 0 & 1 \\
           \end{array} \),
\end{equation}
which has $1$ on the main diagonal, $a_{1,2}=  \cdots = a_{i,i+1}
= \cdots = a_{m-1,m} = \1$, $a_{m,1} = \1$, and all  other entries
are $0$.

$\sbA$ is  a singular matrix; this seen easily by taking the sum
of its columns, which are dependent and thus $\sbA$ is singular,
cf. Corollary \ref{cor:nRank}). On the other hand, each of whose
diagonal minor $M_i = A_{i,i}$ (obtained by deleting the column
$i$ and the row $i$) is nonsingular. To see the latter, consider
the digraph $G_A$ of $\sbA$ which has two $m$-multicycles. One
$m$-multicycle is given by the $m$ self loops $\sig_i := (i,i)$,
all labeled $1$, while the other is given by the sequence
$$\sig := (1,2), (2,3), \dots, (i,i+1), \dots, (m-1,m), (m,1)$$ of
edges, all labeled $\1$. The subgraph $G_{M_i}$ of $G_A$
corresponding to a minor $M_{i}$ is obtained by deleting the
vertex $i$ from $G_A$, and thus deleting the single self loop
$(i,i)$ and two edges of $\sig$ emerging and terminating at $i$.
Accordingly, $G_{M_i}$ has a unique $(m-1)$-multicycle, composed
from $(m-1)$ self loops $\sig_i$, all labeled $1$, and thus the
minor $M_{i}$ is nonsingular by Remark \ref{rmk:perma}.
\end{example}

The above examples show that an $\sbool$-representation can often
be  reduced further.

%

\begin{theorem}\label{thm:hetByCirc} When each basis of
a hereditary collection $\H= (E, \tH)$ is  contained in a
circuite, $\H$
 can be represented in terms of its
circuits.
\end{theorem}

\begin{proof} Let $C_1, \dots, C_\ell \in \tC(\H)$ be the circuits of $\H$.
Given a circuit $C_i$ with $|C_i| = m_i$ we construct the $m_i
\times n$ matrix $\sbD^{(i)} := \sbA(C_i)$ whose submatrix
$\cl{\sbD^{(i)}}{C_i}$ is of the Form \eqref{form:2} and all its
other entries are $\1$. As shown above, $\cl{\sbD^{(i)}}{C_i}$ is
an $\sbool$-representation of the uniform hereditary collection
$U_{m_i-1, m_i}$ and thus, by Example \ref{ex:Umn}, also of the
circuit $C_i$, which in the view of Example \ref{exp:1.2}.(c) is
just $\H_i := (C_i,\Pow(C_i) \sm \{ C_i \})$.

 The proof is then completed
 by stacking these matrices $\sbA(C_i)$ (each corresponds to
a different circuit $C_i$ of $\H$) one over the other, with
respect to their columns labeling, and applying the same argument
as in the proof of Theorem \ref{thm:hdCol}.
\end{proof}

Under the condition of Theorem \ref{thm:hetByCirc}, we have the
following upper bound for the size of $\sbool$-representations of
a hereditary collection:
\begin{equation}\label{eq:Hbound} m \leq \min \bigg\{ \sum_{J \in
\tB(\H)} |J|,  \sum_{C \in \tC(\H)} |C| \bigg \}.
\end{equation}

\subsection{Examples of $\sbool$-representations for matroids} In
this section we provide examples of $\sbool$-representations for
some well  known matroids. (See \cite{oxley:matroid} for further
explanation of the notation.)

\begin{example} Let $\M$ be the matroid $M(K_4)$; that is the matroid corresponding to the following
diagram (dependent 3-subsets correspond to $3$ colinear points in
the diagram):
$$
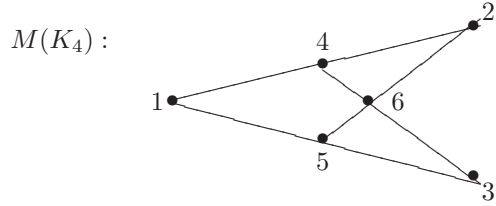
\xy 
(5,18)*+{M(K_4):}, (18,10)*+{1},(62,-2)*+{3},(62,22)*+{2},
(40,2)*+{5}, (40,18)*+{4}, (50,10)*+{6},
(20,10)*+{\bullet},(60,0)*+{\bullet},(60,20)*+{\bullet},
(40,5)*+{\bullet}, (40,15)*+{\bullet}, (46,10)*+{\bullet},
(19,10)*+{}; (61,20)*+{}; **\crv{}, (19,10)*+{}; (61,-1)*+{};
**\crv{},(39,15)*+{}; (61,-1)*+{}; **\crv{}, (39,4)*+{};
(61,21)*+{}; **\crv{},
%
\endxy$$
i.e., the matroid over $6$ elements where all the $3$-subsets are
independent expect:
$$ \{1,2,4\}, \{1,3,5\}, \{ 2,5,6\}, \{ 3,4,6\}.$$
 $\M = M(K_4)$ is $\sbool$-representable by the matrix
$$ \blA(\M) =
     \( \begin{array}{ccccccc}
       \1 & \1 & \1 & 0 & 0 & 1 \\
       1 & 0 & 0 & 1 & 1 & 1 \\
       0 & 1 & 0 & 1 & 0 & 1 \\
       0 & 0 & 1 & 0 & 1 & 1 \\
       \end{array} \).
$$
\end{example}

\begin{example} The matroid $\tW^3$, corresponding to the diagram:

$$ 
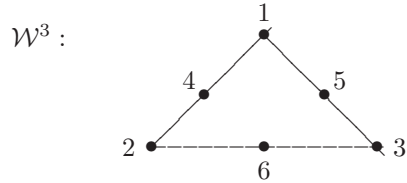
\xy  (10,15)*+{\tW^3:}, (22,0)*+{2},(58,0)*+{3},
 (40,-3)*+{6},(50,9)*+{5},
(30,9)*+{4}, (40,18)*+{1},
(25,0)*+{\bullet},(55,0)*+{\bullet},
 (40,0)*+{\bullet},(48,7)*+{\bullet},
(32,7)*+{\bullet}, (40,15)*+{\bullet},
(24,-1)*+{}; (41,16)*+{}; **\crv{}, (25,0)*+{}; (55,0)*+{};
**\crv{},(39,16)*+{}; (56,-1)*+{}; **\crv{},
%
\endxy$$
is $\sbool$-represented by the matrix
$$ \blA(\tW^3) =
     \( \begin{array}{ccccccc}
       \1 & \1 & \1 & 0 & 0 & 1 \\
       1 & 0 & 0 & 1 & 1 & 0 \\
       0 & 1 & 0 & 1 & 0 & 1 \\
       0 & 0 & 1 & 0 & 1 & 1 \\
       \end{array} \).
$$
\end{example}

\section{Boolean representations}

In this section we study boolean representations; these
representations are a special case of $\sbool$-representations
provided  by boolean matrices. Recall that we write
$\bool$-representations for the boolean representations, and say
that a hereditary collection $\H$ is $\bool$-representable if it
has a $\bool$-representation.


\subsection{Graphic matroids} We begin with the classical concept
of Whitney for a connection between matroids and graphs,  see
\cite{Oxley03whatis,whitney}.

Given a finite graph $\Gr := (\gV,\gE)$ with vertex set $\gV$ and
set of  edges $\gE$ ($\Gr$ might have multiple edges), we consider
the $|\gV| \times |\gE|$ incidence matrix  $\inA(\Gr) :=
(a_{i,j})$ with entry $a_{i,j} =1$ if the vertex $v_i$ is an end
point of the edge $e_j$ and $e_j$ is not a self-loop, otherwise we
set $a_{i,j} =0$. For example, the incidence matrix  of the graph
$$
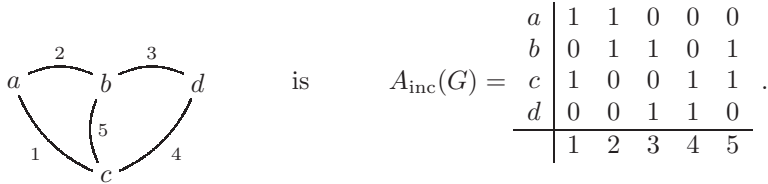
\xymatrix{
    a      \ar@/_/@{-}[dr]_{1 }  \ar@/^/@{-}[r]^2 &
      b \ar@/^/@{-}[r]^3   & d \ar@/^/@{-}[dl]^4\\
             & c \ar@/^/@{-}[u]_5 & \\
 } \qquad \quad \text{is} \qquad \quad
 \inA(\Gr) = \begin{array}{c|ccccc}
                 a & 1 & 1 & 0 & 0 & 0 \\
                 b & 0 & 1 & 1 & 0 & 1 \\
                 c & 1 & 0 & 0 & 1 & 1 \\
                 d & 0 & 0 & 1 & 1 & 0 \\ \hline
                   & 1 & 2 & 3 & 4 & 5  \\
                \end{array} \ .
               $$
For ease of exposition, throughout we assume  that $\Gr$ is a
connected graph. Note that now the rows of the matrix $\inA(G)$
are labeled by the vertices $\gV$, and columns are labeled by the
edges $\gE$ of the graph~$\Gr$.

By this construction we see that each column of the matrix
$\inA(\Gr)$ has either two or no  $1$-entries, and multiple edges
introduce identical columns. Accordingly, without loss of
generality, we may consider the ground set  $E_G$ as a collection
of $2$-subsets of $V_G$. A matroid constructed by this way is
called \textbf{graphic matroid} and we denote it $\inM(\Gr)$.

Let $\f2A : =  \inA(\Gr) $ be the incidence matrix $\inA(\Gr)$ of
the graph $G$, considered as a matrix over the field $\bF_2$ of
characteristic $2$. Recall that $\M(\f2A)$ denotes the vector
matroid of the matrix $\f2A$ over the field $\bF_2$.

\begin{theorem}[{\cite[Theorem
2.16]{Oxley03whatis}}]\label{thm:f2g} The independent subsets of
the  vector matroid $\M(\f2A)$ correspond to subsets of edges of
$G$ that do not contain a cycle and $\M(\f2A) =
\inM(G)$.\end{theorem}

Accordingly,  if $G$ is a connected graph, then the bases of
$\inM(G)$ are precisely the edge subsets of the spanning trees of
$G$ and, if $G$ has $\ell$ vertices, each spanning tree has
exactly $\ell - 1$ edges, so $\rnk(\inM(G)) = \ell - 1$.


\begin{proposition}\label{2.1.g:prop} Let $\Gr := (\gE, \gV)$ be a connected
graph and let $\inA(\Gr)$ be its adjacency matrix as described
 above. Considering $\inA(\Gr)$ as a matrix over
$\bF_2$ (the finite field of 2-elements) or over the boolean
\semiring \ $\bool$ gives the same matroid of rank $|\gV|-1$ whose
collection of bases correspond to the edges of the spanning trees
of~$\Gr$.
\end{proposition}

\begin{proof}
Write $\f2A$ and $\blA$ for the incidence matrix  $\inA(G)$
considered as a matrix over $\bF_2$ and $\bool$, respectively, and
let $\M(\f2A)$ and $\M(\blA)$ be the corresponding vector
matroids. Theorem \ref{thm:f2g} gives us the correspondence
between the bases of $\M(\f2A)$ and the spanning trees of $\Gr$,
so we need to
 prove that the bases of $\M(\blA)$ are in one-to-one
 correspondence with the spanning trees of $\Gr$.
 (This will also prove that $\M(\blA)$ is indeed a matroid.)

Suppose $|\gV| = \ell$ and let $T_1 \subseteq \gE$ be a spanning
tree of $\Gr$. Then $T_1$ has $\ell-1$ edges. Consider the $\ell
\times (\ell-1)$ submatrix $B_1 := \cl{\blA}{T_1}$ of $\inA(G)$,
cf. Notation \ref{nott}, corresponding to $T_1$  and having two
$1$-entries in each column by construction. Pick a leaf vertex
$i_1 \in \gV$, which belongs to a unique edge $ (i_1,j_1)$ of
spanning tree $T_1$, and erase $i_1$ and its connecting edge
($i_1,j_1)$ from $T_1$ to obtain the subtree $T_2$ of $T_1$ (which
clearly is connected, and has $\ell-2$ edges). This deletion is
expressed by erasing the $i_1$'th row of $B_1$, which is an
$(\ell-1)$-marker, and the column $j_1$ (corresponding to the edge
$(i_1,j_1)$) of $B_1$. Denote this matrix corresponding to subtree
$T_2$ by $B_2$ and let $D_1$ be the matrix composed of the
$i_1$'th row of~ $\cl{\blA}{T_1}$.

We repeat this process recursively, erasing at each step a new
leaf vertex  $i_k$ from the tree $T_k$ having $\ell-k$ edges,
expressed as a deletion of the row and the column corresponding to
vertex the $i_k$ and its connecting edge $ (i_k,j_k)$ in the
matrix $B_k$, and joining the $i_k$'th row of $\cl{\blA}{T_1}$ to
the matrix  $D_{k-1}$ from below. At the end of this process,
after $\ell -1$ steps, we obtain the triangular $(\ell -1) \times
(\ell -1)$ matrix $D_{\ell -1}$, which by construction is of the
Form \eqref{eq:trgform} -- a nonsingular matrix by Lemma
\ref{lem:2.1.f}. $D_{\ell-1}$  is a submatrix of $\cl{\blA}
{T_1}$, up to permuting of rows, and thus an $(\ell -1) \times
(\ell- 1)$ witness. Therefore, the columns of $\cl{\blA}{T_1}$ are
independent.

Conversely, suppose $X \subseteq \gE$ with $|X| = k$  and assume
that the columns of $\cl{\blA}{X}$ are independent, namely
$\rnk_{\sbool}(\cl{\blA}{X}) = k$. But then, by Proposition
\ref{prop:rank}, $\rnk_{\bF_2}(\cl{\f2A}{X}) \geq k$ which implies
$\rnk_{\bF_2}(\cl{\f2A}{X}) = k$, since $|X| = k$. Hence, by
Theorem \ref{thm:f2g}, $\cl{\f2A}{X}$ corresponds to a spanning
tree of $\Gr$ and so does $\cl{\blA}{X}$.
 \end{proof}

\subsection{Boolean representations of hereditary collections}
Our first result in this section ties boolean representations to
the point replacement property, PR, (cf. Definition~
\ref{def:ratroid}) for an arbitrary hereditary collection.

\begin{theorem} If a hereditary collection $\H = (E,\tH)$ has a  $\bool$-representation,
then $\H$ satisfies PR.
\end{theorem}

\begin{proof} The case of $|E| = 0$ is obvious, so throughout we assume that $| E | > 0$.
 Suppose $\H$ is   $\bool$-representable by the matrix $\blA := \blA(\H)$, we need to
verify  Proposition \ref{prop:het}.(ii), that is  $$\text{$\{p\},
X \in \tH$, with $p \notin X \neq \emptyset$,
    implies $\exists x \in X$, such that $X - x + p \in
    \tH$}.$$

Let $P = \{ p \}$ and suppose   $|X| = k$. Reordering
independently the columns and rows of $\blA$ we may assume that
$\blA$ has a $k \times k$ witness  $\clrw{\blA}{X}{Y}$ of the
triangular form \eqref{eq:trgform} for some  $Y \subseteq
\Rw(\blA(\H))$ with $|Y| =k$. (We use again Notation \ref{nott}
for submatrices.)  Clearly $\cl{\blA}{ P }$ has a nonzero entry,
since $P \in \tH$ and $P \neq \emptyset$.

 Assume first that $\clrw{\blA}{P}{Y}$ has a nonzero entry,  and let $\ell$ be the
 smallest index entry of the column
 $\clrw{\blA}{P}{Y}$ which is not zero.   Then, we
are done by interchanging $\cl{\blA}{P}$ with the $\ell$'th column
of $\cl{\blA}{X}$, since we have preserved the triangular form
\eqref{eq:trgform}.  Otherwise, let $Z \supset Y$ be a subset of
rows with $|Z| = |Y|+1$ such that $\clrw{\blA}{P}{Z}$ has a single
nonzero entry. Then $\clrw{\blA}{X \cup P}{ Z}$ is a $(k+1) \times
(k+1)$ witness, and thus any subset of $k$ columns of $\cl{\blA}{X
\cup P}$ is independent by Theorem
\ref{thm:regularityToIndependent}. Namely, any column of
$\cl{\blA}{X}$ can be replaced by $\cl{\blA}{P}$, preserving the
independence relations.
\end{proof}

We are now ready for another main result of this paper relating to
matroids.

\begin{theorem}\label{thm:boolFRep} If a matroid $\M$ is $\bF$-representable, for some
field $\bF$,  then $\M$ is also $\bool$-representable.
\end{theorem}

\begin{proof} Let  $\M$ be a matroid of rank $m$ and suppose it is
$\bF$-representable by the matrix $\fA:= \fA(\M)$. We may assume
that $\fA$ has rank $m$ since otherwise by row operations
(including subtraction, since $\bF$ is a field) we can bring $\fA$
to have exactly $m$ nonzero rows. If $m = 0$ we are done, so
throughout we assume that $m > 0$.

 Let $J_1, \dots, J_\ell$ be the bases of $\M$. Given a basis
 $J_i$, then
 $\fA$ has an $m\times m$ witness $\cl{\fA}{J_i}$ (see Notation
 \ref{nott}). Applying classical row
operations to $\fA$, including subtraction, 
we can reduce $\fA$ so that the  submatrix $\cl{\fA}{J_i}$ is a
triangular matrix, i.e., $1$ over all the main diagonal and $0$
above the diagonal; the entries of $\cl{\fA}{E \sm J_i}$ can take
arbitrary values.  We denote this matrix by $\fA^{(i)}$. We repeat
the same process with respect to each basis $J_i$, $i =1, \dots,
\ell$, to obtain the $m\times n$ matrices $\fA^{(i)}$ over $\bF$.

We construct the $ m \ell \times n$ matrix $\fB$ by
\emph{stacking}  the $\ell$ matrices $\fA^{(i)}$ by the indexing
order $i= 1,\dots, \ell$. Note that, since we have used only row
operations to obtain the matrices $\fA^{(i)}$'s, as well as
duplications of rows, the columns of $\fB$ satisfy exactly the
same linear dependence relations which were satisfied by the
columns of $\fA$. Thus $\fB$ is also an $\bF$-representation of
$\M$.

We introduce the boolean matrix $\blB$ obtained from  $\fB$ by
setting all the nonzero entries of $\fB$ to  $1$ and leaving  the
$0$'s as they were. In the same way, we obtain the boolean
matrices $\blB^{(i)}$ from $\fA^{(i)}$. (Of course stacking  the
boolean matrices $\blB^{(i)}$ by the indexing order yields $\blB$
again.)

Let $Y_i \subseteq \Rw(\fB)$ be the rows of $\fB$ corresponding to
the matrices $\fA^{(i)}$. Therefore,  $\clrw{\fA^{(i)}}{J_i}{Y_i}$
is an $m \times m$ witness and so does
$\clrw{\blB^{(i)}}{J_i}{Y_i}$, since is of the Form
\eqref{eq:trgform}. Using Corollary \ref{cor:sinDep}, it easy to
see that by this construction the columns of $\cl{\blB}{J_i}$ are
independent since $\clrw{\blB^{(i)}}{J_i}{Y_i}$ is an $m \times m$
witness contained in $\cl{\blB}{J_i}$.

To complete the proof we need to show that we have not introduced
new independent column subsets other than the ones we had in
$\fA$. Suppose $X \subseteq E$, with $|X| = k$,  and assume that
the columns of $\cl{\blB}{X}$ are independent. Thus,
$\cl{\blB}{X}$ contains a $k \times k$ witness  $\blD :=
\clrw{\blB}{X}{Y}$, for some $Y \subseteq \Rw(\blB)$ with $|Y| =
k$, which by Lemma \ref{lem:2.1.f} is of the Form
\eqref{eq:trgform}, up to permutation. But  the witness $\blD$ was
obtained from a submatrix  $\fD$ of $\fB$ by changing every
nonzero entry to be~$1$. Thus, $\fD$ is also of the Form
\eqref{eq:trgform}, up to permeating of rows and columns, where
the elements below the diagonal are now elements of $\bF$. This
means that the matrix $\fD$ is a $k \times k$ (field) witness,
also a submatrix of $\fB$. Namely the columns of $\cl{\fB}{X}$
were independent, and thus also in the initial
$\bF$-representation $\fA(\M)$ of $\M$.

Therefore  this shows that $\blB $  is a proper boolean
representation of the matroid $\M$.
\end{proof}

Having Theorem \ref{thm:boolFRep} at hand, we can generalize it
much further.

\begin{theorem}\label{thm:matDecom} Suppose $\M = \M_1 \oplus \cdots \oplus
\M_\ell$, where  $\M_i := (E_i, \tH_i)$, is a matroid, with
disjoint $E_i$'s and each $\M_i$ is $\bF_i$-representable matroid
for some field $\bF_i$. Then $\M$ has a boolean representation.
\end{theorem}

\begin{proof} As proved in Theorem
\ref{thm:boolFRep}, every $\M_i$ is $\bool$-representable, let
$\blA(\M_i)$ be its $\bool$-representation. Then we claim that
the matroid $\M$ has the $\bool$-representation
$$ \blA(\M) := \( \begin{array}{cccccc}
              \blA(\M_1)  & 0 & \ldots &   0 \\
                   0 &  \ddots &  & \vdots  \\
                   \vdots &  \ddots &  \ddots & 0  \\
                   0 & \cdots & 0& \blA(\M_\ell) \\
                 \end{array}
\).$$

Given any subsets $X_1 \in \tH_1, \dots, X_\ell \in \tH_\ell$,
where $X_i$ can be empty,  clearly the submatrix $\cl{\blA}{
\bigcup_i X_i}$ is of rank $\sum_i |X_i|$ by construction. On the
other hand, suppose that the columns of the submatrix
$\cl{\blA}{X}$ are independent and write $X = X_1 \ds {\dot \cup}
\cdots \ds {\dot \cup} X_\ell$, with $X_i \subseteq E_i$ (could be
$\emptyset$). Let $Y_i \subset \Rw(\blA(\M))$ be the subset of
rows corresponding to $\blA(\M_i)$ in $\blA(\M)$. Then, since
$\cl{\blA}{X}$ is independent, any of its column subsets is also
independent and in particular each $\cl{\blA}{X_i}$. Since we have
not introduced any new $1$-entries in $\blA(\M)$, expect those of
the matrices $\blA^{(i)}:= \blA(\M_i)$, the columns of
$\cl{\blA}{X_i}$ are independent, as well as
$\clrw{\blA}{X_i}{Y_i} = \cl{\blA^{(i)}}{X_i}$, and thus $X_i \in
\tH_i$.
\end{proof}

We can conclude the following immediately:

\begin{corollary} There are hereditary collections (and in particular matroids) which
are  \textbf{not}  $\bF$-representable over any field $\bF$ but do
have a $\bool$-representation.
\end{corollary}

As an example for the corollary consider the well known matroids,
the Fano matroid $F_7$  and the non-Fano matroid $F_7^-$ (see
Section \ref{ssec:Fano} below for explicit description). It is
known that $F_7$ is $\bF$-repressible iff $\bF$ is a field of
characteristic $2$, while $F_7^-$ has a field representation  iff
  $\bF$ is of characteristic $\neq 2$, cf. \cite[Proposition
5.3]{Oxley03whatis}. Accordingly,  the direct sum $F_7 \oplus
F_7^-$ of these matroids is not representable over any field, cf.
\cite[Corollary 5.4]{Oxley03whatis}, but it is
$\bool$-representable by Theorem~\ref{thm:boolFRep}.

\subsection{Fano and non-Fano matroids}\label{ssec:Fano}
Let
$$ A_7 := \( \begin{array}{ccccccc}
       1 & 0 & 0 & 1 & 1 & 0 & 1 \\
       0 & 1 & 0 & 1 & 0 & 1 & 1 \\
       0 & 0 & 1 & 0 & 1 & 1 & 1 \\
       \end{array} \)
       $$
be the matrix whose columns are labeled by the ground set $E =
\{1, \dots, 7 \}$.

Considering  $A_7$ as a boolean matrix, written $\bool(A_7)$, the
columns of $\bool(A_7)$ are all the possible nonzero boolean
 3-tuples of the $3$-space  $\bool^{(3)}$.
The independent column subsets of $A_7$ correspond to the
independent subsets of vectors of $\bool^{(3)}$, and thus
introduce a $\bool$-vector hereditary collection (cf.
Definition~\ref{defn:VecHC}), denoted by $\H(\bool^{(3)})$ and
identified with  $\bool(A_7)$. Abusing notation we write
$\bool(A_7)$ for this hereditary collection,  but no confusion
should arise.

A direct computation shows that the independent subsets of $E$,
determined by $\bool(A_7)$, are, the empty set, all the subsets
with 1 or 2 elements, and all the  3-subsets of $E$ except the
following ten 3-subsets:
$$
\begin{array}{lllll}
\{ 1,2,4\}, & \{1,3,5 \}, &  \{ 1,6,7\}, & \{2,3,6 \}, & \{2,5,7
\}, \\[1mm]
 \{3,4,7 \}, &  \{4,5,6 \}, & \{4,5,7 \}, &  \{4,6,7 \},& \{5,6,7
\}.\end{array}
$$
$\bool(A_7)$ satisfies PR but is not a matroid, since considering
the  two bases
$$ B_1 = \{1,5,7 \} \qquad \text{and} \qquad B_2 = \{2,4,7 \}$$
we see that the element  $5$ from $B_1$ can  replace neither $2$
nor $4$ and preserve BR. The hereditary collection
$\H(\bool^{(3)})$ has $35-10= 25$ bases.

Next, consider $A_7$  as a matrix over a field $\bF_2$ of
characteristic  $2$ to  obtain the Fano matroid $F_7 :=
\bF_2(A_7)$, described by the diagram
$$
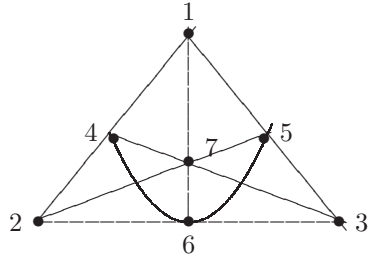
\xy (17,0)*+{2},(63,0)*+{3},
 (40,-3)*+{6},(53,12)*+{5},
(27,12)*+{4}, (40,28)*+{1}, (43,10)*+{7},
(20,0)*+{\bullet},(60,0)*+{\bullet},
 (40,0)*+{\bullet},(50,11)*+{\bullet},
(30,11)*+{\bullet}, (40,25)*+{\bullet}, (40,8)*+{\bullet},
(19,-1)*+{}; (41,26)*+{}; **\crv{}, (20,0)*+{}; (60,0)*+{};
**\crv{},(39,26)*+{}; (61,-1)*+{}; **\crv{},(19,0)*+{};
(51,12)*+{}; **\crv{},(29,12)*+{}; (61,0)*+{};
**\crv{},(40,0)*+{}; (40,25)*+{}; **\crv{}
,(29,13)*+{}; (51,13)*+{};  **\crv{(40,-13)}
\endxy$$
(See \cite{oxley:matroid} for more explanation of the notation.)

The bases of the matroid $F_7$ are all the 3-subsets of $E$ except
those 3-subsets which lie on a same line (could also be a curved
line); these 3-subsets are:
$$
\begin{array}{lllll}
\{ 1,2,4\}, & \{1,3,5 \}, &  \{ 1,6,7\}, & \{2,3,6 \}, & \{2,5,7
\}, \\[1mm]
 \{3,4,7 \}, &  \{4,5,6 \}. &  &  & \end{array}
$$ So, we  have joined  the three independent 3-subsets $\{4,5,7 \},
\{4,6,7 \}, \{5,6,7 \}$ to those of  $\bool(A_7)$.

The non-Fano matroid, denoted $F_7^-:= \bF_3(A_7)$, is given by
the diagram
$$
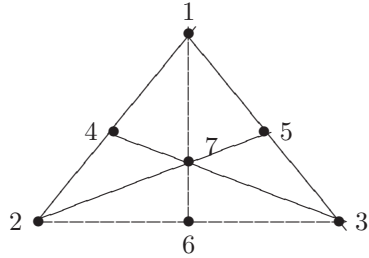
\xy (17,0)*+{2},(63,0)*+{3},
 (40,-3)*+{6},(53,12)*+{5},
(27,12)*+{4}, (40,28)*+{1}, (43,10)*+{7},
(20,0)*+{\bullet},(60,0)*+{\bullet},
 (40,0)*+{\bullet},(50,12)*+{\bullet},
(30,12)*+{\bullet}, (40,25)*+{\bullet}, (40,8)*+{\bullet},
(19,-1)*+{}; (41,26)*+{}; **\crv{}, (20,0)*+{}; (60,0)*+{};
**\crv{},(39,26)*+{}; (61,-1)*+{}; **\crv{},(19,0)*+{};
(51,12)*+{}; **\crv{},(29,12)*+{}; (61,0)*+{};
**\crv{},(40,0)*+{}; (40,25)*+{}; **\crv{}
%
\endxy$$
with  $A_7$ considered as a matrix over a field $\bF_3$ of
characteristic $3$. The bases of $F_7^-$ are then all the
3-subsets of $E$ except:
$$
\begin{array}{lllll}
\{ 1,2,4\}, & \{1,3,5 \}, &  \{ 1,6,7\}, & \{2,3,6 \}, & \{2,5,7
\}, \\[1mm]
 \{3,4,7 \}. &  &  &  & \end{array}
$$
The boolean representation  $A_\bool(F_7)$ of $F_7$,  obtained
from $\bool(A_7)$, is given by the matrix
$$ A_\bool(F_7) = \( \begin{array}{ccccccc}
       1 & 0 & 0 & 1 & 1 & 0 & 1 \\
       0 & 1 & 0 & 1 & 0 & 1 & 1 \\
       0 & 0 & 1 & 0 & 1 & 1 & 1 \\
       0 & 1 & 1 & 1 & 1 & 0 & 0 \\
       1 & 0 & 1 & 1 & 0 & 1 & 0 \\
       \end{array} \).
       $$
       One sees that the restriction of $A_\bool(F_7)$ to
       the three upper rows is the matrix $\bool(A_7)$.
       It easy to verify that the two bottom lines of
       $A_\bool(F_7)$ provides the independence of the $3$-subsets $\{4,5,7 \}, \{4,6,7 \}, \{5,6,7
       \}$ and has no influence on the other dependent 3-subsets.

By the same argument we obtain from  the matrix $A_\bool(F_7)$,
the $\bool$-representation $A_\bool(F_7^-)$ of $F_7^-$:
$$ A_\bool(F_7^-) = \( \begin{array}{ccccccc}
       1 & 0 & 0 & 1 & 1 & 0 & 1 \\
       0 & 1 & 0 & 1 & 0 & 1 & 1 \\
       0 & 0 & 1 & 0 & 1 & 1 & 1 \\
       0 & 1 & 1 & 1 & 1 & 0 & 0 \\
       1 & 0 & 1 & 1 & 0 & 1 & 0 \\
       1 & 1 & 1 & 1 & 0 & 0 & 1 \\
       \end{array} \),
       $$
where now we had to add an  additional row to make the column set
$\{4,5,6 \}$ independent without changing the existing dependence
relations of the columns of  $A_\bool(F_7)$.

 The matrix
$$ A(F_7 \oplus F_7^-) =
\left(%
\begin{array}{c|c}
  A_\bool(F_7) & 0 \\
  & \\ \hline
   & \\
 0  & A_\bool(F_7^-) \\
\end{array}%
\right) $$ gives a boolean representation of the direct sum $F_7
\oplus F_7^-$, which is  a matroid that is known not to be
representable over any  field, cf. \cite[Corollary
5.4]{Oxley03whatis}.

\subsection{Boolean representations
and matching in bipartite graphs}  Given a hereditary collection
$\H = (E,\tH)$, with $|E| = n$, that has a boolean-representation
by an $m \times n$ matrix $\blA := \blA(\H)$, we associate the
matrix $\blA = (a_{i,j})$ with the bipartite graph $G := (\gV'
\cup \gV'', \gE)$ having $m +n $ vertices $\gV' \cup \gV''$, i.e.,
$|\gV'| = m$ and $|\gV''| = n$, and edges $(i',j'') \in \gE$,
where $i' \in \gV'$ and $j'' \in \gV''$, iff $a_{i',j''} =1$ in
the matrix~ $\blA$. (See \cite[Chapter 2]{murota}, and also
\cite{Jungnickel}, \cite{oxley:matroid}, for more details.)

Recall that by condition WT (cf. Definition \ref{def:ratroid}) a
subset $X \subseteq E$ (realized also as a column subset $X
\subseteq \Cl(\blA)$) is independent iff there exists a row subset
$Y \subseteq \Rw(\blA)$ with $|X| = |Y|$ such that the submatrix
$\clrw{\blA}{X}{Y}$ is a witness (cf. Notation \ref{nott}).
Abusing the notation and considering respectively $X$ and $Y$ also
as vertex subsets of $\gV''$ and $\gV'$ in the graph $G$, we see
that $X$ is independent in $\H$ iff there exists a vertex subset
$Y \subseteq \gV'$ so that $G$ has a \textbf{unique matching} of
$X$ onto $Y$. That is,  there is one and only one matching of $X$
onto $Y$ in the graph $G$ of $\blA$.

 If we consider the vertex subsets
of $\gV''$ which have some matching (not necessarily unique) to
some vertex subsets $Y$ of $\gV'$ with respect to $G$, we obtain
the usual transversal matroid of $G$ (see \cite{oxley:matroid}).
When we restrict to those subsets of $\gV''$ having unique onto
matchings in the above sense, we obtain, in general, a hereditary
collection over the same ground structure with less independent
subsets of $\gV''$ which satisfy PR (cf. Definition
\ref{def:ratroid}) but is not necessarily a matroid. However,
since every traversal matroid is representable by some field (cf.
\cite{oxley:matroid}), every transversal matroid corresponds to
the unique onto  matchings of some, in general different,
bipartite graph by Theorem \ref{thm:boolFRep}.

   Thus,  boolean representations of hereditary collection
   (which must satisfy PR) have a strong connection to classical matching theory.


\section{Open questions}\label{sec:openQ}

We open with the main questions first.

\begin{question}\label{q:4.1.a}
Do all the matroids have to have  a $\bool$-representation? \\ If
not, which matroids do have $\bool$-representations?
\end{question}

\begin{question}\label{q:4.1.b}
Do all the hereditary collections satisfying PR have
$\bool$-representations? \\ If not, which such hereditary
collections have $\bool$-representations?
\end{question}

\begin{question}\label{q:4.1.c}
Which hereditary collections (matroids) with $|E| =n$ are
$\sbool$-representable by $m \times n$ matrices where $m \leq n$?

\end{question}

Let $\bK$ be a commutative semiring, for a given hereditary
collection $\H$ we define the function $$ s_\bK: \ds \H \to  \bN
\cup \{\infty \}$$ whose value in $\bN$ is  the minimal number of
rows of any $\bK$-representation of $\H$ and  is $\infty$ when
$\H$ is not $\bK$-representable. Clearly $s_\bK(\H) \geq \rnk(\H)$
for any $\H$. For example, we showed in Example \ref{exp:U2n} that
$s_\bool(U_{2,n}) \leq n-1$ over the boolean \semiring.

\begin{question}\label{q:4.2.b} Compute $s_\bK(\H)$.
\end{question}

\begin{question}\label{q:4.2.b} Given a hereditary
collection, is $s_\bK(\H)$ computable in the computer science
sense for $\bK$ boolean, superboolean, max plus,  etc.?
\end{question}


When $\bK$ is a field and $\M$ is a matroid $s_\bK(\M)$
    is either $\infty$ or equals to the rank of $\M$. So the main
    questions are the  Rota's conjectures \cite{Oxley03whatis,oxley:matroid}.

\begin{question}\label{q:4.2.b}  What are the possible values of $s_\bK(\H)$
when $\H$ is hereditary collection which also satisfies PR  and
$\bK$ is the boolean or the superboolean \semiring? What are the
lower and upper bounds for $\sbool$-representations?
\end{question}

Theorem \ref{thm:hdCol} shows that $s_\sbool(\H) < \infty$ for any
hereditary collection $\H$  and  a  bound on  $s_\sbool(\H)$ is
given in Equation~\eqref{eq:upperBond}.
\begin{question}\label{q:4.2.b} When is $s_\sbool(\H) \leq |E|$, or
$s_\sbool(\H) = \rnk(\H)$? Find a lower bound for
$\sbool$-representations with respect to the rank of $\H$.
\end{question}

\section*{Appendix A. Tropical and supertropical algebra}

 A \textbf{semiring} $(R,+,\cdot \, , \rzero, \rone)$, written $(R,+,\cdot \, )$ for short,  is a set
$R$ endowed with two  binary operations $+$ and $\cdot \, $,
addition and multiplication, respectively, and distinguished
elements $ \rzero$ and $\rone$,
 such that  $(R,\cdot \, , \rone)$ is a monoid and $(R,+,\rzero)$
 is an commutative
monoid satisfying distributivity of multiplication over addition
on both sides, and such that $\rzero \cdot a = a \cdot \rzero =
\rzero$ for every $a \in R$ \cite[\S8-\S9]{qtheory}.
 A (two sided) semiring \textbf{ideal} $\mfa$ of $R := (R,+,\cdot \, , \rzero, \rone)$ is
 an additive  subgroup of $(R, +, \zero)$, i.e., $a , b \in \mfa$
 implies $a+  b \in \mfa$,
 for which $x a \in \mfa$ and $a x  \in \mfa$ for every  $x \in R$ and $a \in
\mfa.$

 A semiring $R$ is additively
\textbf{idempotent} if $a + a = a$ for every $a \in R$. Letting $R
^ \times := R \sm \{ \rzero \}$,  when $( R ^ \times,\cdot \, ,
\rone)$ is an Abelian group, we say that $R$ is a
\textbf{semifield}. The notion of semifield does not have a formal
consistent definition in the literature, for that reason we
preserve the terminology of semirings along this paper.

\subsection*{A.1. Tropical structures} Traditionally, tropical algebra
takes place over the tropical (max-plus) \semiring \
$\Real_{(\max, + )} := (\Real \cup \{ -\infty\}, \max, +)$, the
real numbers together with the formal element $-\infty$ equipped
with the operations of maximum and summation, providing
respectively the \semiring \ addition and the multiplication. Over
this setting $\zero:= -\infty$ is the zero element of the
\semiring \ and the number $0$ is the multiplicative unit, denoted
$\one$. Dually, one  has the min-plus \semiring \ $\Real_{(\min, +
)} := (\Real \cup \{ \infty\}, \min, +)$, where now  $\zero:=
\infty$. (Both structures are semifields according to the above
definition.)

\begin{remark*} The boolean \semiring \ is embedded naturally in the
tropical \semiring \ $\Real_{(\max, + )}$, the embedding $\varphi:
\bool \hookrightarrow \Real_{(\max, + )}$ is given by
$$ \varphi : 1 \mapsto 0, \qquad \varphi:0 \mapsto -\infty.$$
\end{remark*}

The max-plus \semiring \ $\Real_{(\max, + )}$ is a special case of
an (additive) idempotent \semiring \ \cite{Litvinov2005}, i.e., a
\semiring \ in which $a+a =a$ for any $a \in \Real$. In general,
one may replace the \semiring \ $\Real_{(\max, + )}$ by an
idempotent \semiring \ $R := (R, +, \cdot \; )$ satisfying the
\textbf{bipotence property}
$$ a + b \in \{ a , b \}, \qquad \text{ for any } a,b \in R.$$
(Note that $R$ is then ordered by the role $a > b \Leftrightarrow
a + b = a$.) We call such a \semiring \ a \textbf{bipotent
\semiring}; for example the boolean \semiring , as well as  the
tropical semiring, is a  bipotent \semiring.

Bipotent semirings arisen naturally from (totally) ordered
cancellative monoids in the following way. Given an ordered monoid
$(M , \cdot \ )$, we adjoin $M$ with the formal element $-\infty$,
declaring $-\infty < a$ for any $a \in M$. Then, the addition of
$M \cup \{ -\infty \}$ is defined as
$$  a + b = \max\{ a,b \} \qquad \text{ for any } a,b \in M, $$
where the multiplication is given by the original monoid operation
of $M$, extended  with $a (-\infty) = (-\infty) a = -\infty$. By
this construction,  when the monoid $M$ is an Abelian group, the
obtained semiring is a semifield.

\subsection*{A.2. Supertropical structures}
A \textbf{supertropical semiring} is a semiring $R := (R, +, \cdot
\;, \tGz, \nu )$  with a distinguished ideal $\tGz$, called the
\textbf{ghost ideal}, and a semiring projection $\nu: R \to \tGz$,
satisfying the axiom (writing $a^\nu$ for $\nu(a)$):
$$ \begin{array}{lll}
\text{\emph{Supertropicality}:}  &  a+b = a^\nu  &  \text{ if } \
\
a^\nu = b^\nu. \\
\text{\emph{Bipotence}:} &  a+b \in \{a,b \} & \text{  if } \ \
a^\nu \neq  b^\nu.
 \end{array}
$$
Thus $\tGz$ is equipped with the natural partial order $ a^\nu \ge
b^\nu \ \text{iff}\  a^\nu + b^\nu =a^\nu,$ which is incorporated
into the semiring structure, written $$ a
>_\nu b \qquad\text{iff}\qquad  \nu(a) > \nu(b).$$

Note that, by definition, in the supertropical arithmetics  we
have
$$ \one +
 \one =  \one  +
 \one +  \one = \ds \cdots = \one  +
 \one + \cdots  + \one = \one^\nu$$  for any arbitrary number of summands greater than two,
 and furthermore $$\text{ $a+ a
 = a +a +a = \ds \cdots = a + a + \cdots +a =
 a^\nu,$ \qquad for any $a \in R$;}$$
therefore a supertropical semiring is not idempotent. Accordingly
we also have $$\one + \one^\nu = \one^\nu + \one^\nu = \ds \cdots
= \one^\nu + \cdots + \one^\nu = \one^\nu,$$ and the same for $a +
a^\nu = a^\nu$.

 A \textbf{supertropical semifield}  $F := (F,  +,
\cdot \;, \tGz, \nu)$ is a supertropical semiring with a totally
ordered ghost ideal~$\tGz$, such that $\tT := F \setminus \tGz$ is
an Ablian group, called the group of \textbf{tangible elements},
for which the restriction $\nu| _\tT : \tT \to \tG$ is onto.

\subsection*{A.3. Supertropicalization}
Any  bipotent semiring $R = (R, + , \cdot \; )$ can be
``supertropicalized'' as following, cf.
\cite{IzhakianRowen2007SuperTropical}. Consider the disjoint union
$$T(R) := \tT  \ \dot \cup \ \{ \zero \} \ \dot \cup \  \tG,$$ with $\tT =
\tG = R \sm \{ \zero \}$. Denote the members of $\tG$ by $a^\nu$,
for each  $a \in \tT$,  and let $\nu : T(R) \to \tG $ be the map
sending $a \mapsto a^\nu$ and be the identity on $\tG \cup \{
\zero\}$. Writing $x,y$ for general elements in $R$, the new
semiring operations $\iplus$ and $\idot$, addition and
multiplication respectively,  are then defined as:
$$ x \iplus y = \left\{ \begin{array}{ll}
            x  &    \nu(x)  > \nu (y), \\[1mm]
            y  &    \nu(x) < \nu (y), \\[1mm]
            \nu(x)    & \nu(x) = \nu (y), \\[1mm]
           \end{array} \right.  $$
for any $x,y \in R$, and
            $$
           \begin{array}{ll }
             a \idot b  = a  b   &  a,b \in \tT,\\[1mm]
             a \idot b^\nu = b^\nu \idot a  = (a  b)^\nu   &  a \in \tT, b^\nu \in \tG,\\[1mm]
             x \idot \zero = \zero \idot x = \zero    &  \forall x \in T(R).\\[1mm]
           \end{array}
$$
Then, $T(R) := (T(R), \iplus, \idot , \tGz, \nu )$, with $\tGz$
and the ghost map $\nu : T(R) \to \tG $ as defined above, is a
supertropical semiring.

The ghost ideal in this construction is the copy $R^\nu$ of $R$
and the \textbf{tangible} elements are $\tT = R \sm \{ \zero \}$.
Moreover, the semiring ideal $\tG :=  R ^\nu$ is a semiring by
itself isomorphic to $R$, therefore $\nu$ composed with this
isomorphism provides an epimorphism $T(R) \to R$. When the initial
semiring $R$ is a semifield, then supertropicalization $T(R)$ is a
supertropical semifield.

A major example of the above construction is provided by starting
with the familiar max-plus algebra. Taking the
supertropicalization of the standard tropical semiring
$\Real_{(\max, + )}$ we obtained the extended tropical
semiring~\cite{zur05TropicalAlgebra}
$$\Trop :=  T(\Real_{(\max, + )}) = \Real \cup \{ - \infty \} \cup \Real^\nu,$$ having the
tangibles $\tT  := \Real$ and ghosts  $\tG := \Real^\nu,$ the
ghost map is given by $a \mapsto a^\nu$ for any $a \in \Real$,
where the semiring operations of $\Trop$ are as described above.
Therefore, $\Trop$ can be thought of as the super-max-plus
algebra.

Using the same construction, one sees that the superboolean
\semiring \ is a supertropicalization of the boolean \semiring, in
other words $\sbool = T(\bool)$.

\subsection*{A.4. Supertropical matrix algebra} Given a supertropical semifield $R$, the algebra
of matrices over~$R$ is developed exactly along the same line of
\S\ref{ssec:matrixAlg}, see \cite{IzhakianRowen2008Matrices,
IzhakianRowen2009Equations}, using similar definitions which are
now  taken with respect to the larger ghost ideal of the ground
supertropical semifield. The results from matrix algebra presented
in our exhibition in \S\ref{ssec:matrixAlg}, Theorem
\ref{thm:regularityToIndependent}, Corollary \ref{cor:n+1vectors},
Corollary \ref{cor:nRank},  and Theorem \ref{thm:rnkSing}, are all
valid in general for matrices taking place over any supertropical
semifield,
{\cite{IzhakianRowen2008Matrices,IzhakianRowen2009TropicalRank}}.

\section*{Appendix B. Tropical representations of hereditary collections}

Superboolean representations of hereditary collections can be
performed in a much wider context obtained  by replacing the
ground superboolean semiring  $\sbool$ by a supertropical
semifield  $F$, for example by $\Trop := T(\Real_{(\max, + )})$.
Namely, given an $m \times n$ matrix $\ffA$ over a supertropical
semifield $F$, we associate the ground set $E := E(\ffA)$ to the
set of columns $\Cl(\ffA)$ of $\ffA$, which as usual are realized
as vectors in~$F^{(n)}$. The independent subsets $\tH :=
\tH(\ffA)$ of $E$ are subsets corresponding to column subsets that
are tropically independent of the $n$-space ~$F^{(n)}$, cf.
\cite[Definition 6.3]{IzhakianRowen2008Matrices}. The $F$-vector
hereditary collection $(E(\ffA),\tH(\ffA))$ is denoted $\H(\ffA)$.
A hereditary collection  $\H'$ that is isomorphic to $\H(\ffA)$
for some matrix $\ffA$ over a supertropical semifield~$F$ is
called $F$-representable; the matrix $\ffA$ is called an
$F${-representation} of $\H'$.

There is a natural  semiring embedding $\varphi: \sbool
\hookrightarrow F$, given by $ \varphi : 1 \mapsto \one,$   $
\varphi : \1  \mapsto \one^\nu ,$ $  \varphi:0 \mapsto \zero,$ of
the superboolean \semiring \ $\sbool$  into  an arbitrary
supertropical \semifield \ $F$. Since $\{ \one, \one^\nu, \zero\}
\subseteq F$ is  a sub-semiring of $F$, this embedding induces a
natural matrix embedding $\widetilde \varphi: M_n(\sbool)
\hookrightarrow M_n(F)$, and thus an
 embedding of representations. Therefore,
$\sbool$-representations can be viewed as $F$-representations,
which in a sense are more comprehensive than
$\sbool$-representations, and more generally as
$R$-representations, for~$R$ a (commutative)  supertropical
semiring. Then, by Theorem \ref{thm:hdCol}, we immediately
conclude the following.
\begin{corollary*}
Every hereditary collection is $R${-representable}, over any
  supertropical semiring~ $R$.
\end{corollary*}

Of course one can construct ``richer'' $F${-representations} of
hereditary collections by involving elements of $F$ other than
$\zero$, $\one$, or $ \one^\nu$.

In general, all the results within this paper can be stated in the
context of supertropical \semifield s. However, to make the
exposition clearer, in this paper we have used the simpler
structure of matrices over the superboolean semiring $\sbool$,
aiming to introduce the idea of representing hereditary
collections by considering matrices over semirings.  As have been
shown these matrices are suitable enough for this purpose.

$F${-representations} of matroids, and more genrally of hereditary
collections, will be discussed in details in a future paper.


\end{document}